 \documentclass[pdflatex,sn-mathphys-num]{sn-jnl}
 \usepackage{enumitem}
 \setlist[enumerate]{leftmargin=.5in}
 \setlist[itemize]{leftmargin=.5in}

\usepackage{amsmath,amssymb,bm}
\usepackage{cases}

\usepackage{graphicx}
\usepackage{subcaption} 

\usepackage{rotating}
\usepackage{multirow}
\usepackage{booktabs}
\usepackage{longtable}
\usepackage{tabularx}
\usepackage{placeins}
\usepackage{afterpage}  

\makeatletter
\renewcommand\LT@makecaption[3]{%
	\LT@mcol\LT@cols c{%
		\hbox to\z@{\hss\parbox{\LTcapwidth}{\@makecaption{#2}{#3}}\hss}}}
\makeatother
\DeclareMathOperator*{\dist}{dist}

\DeclareMathOperator*{\Diag}{Diag}

\DeclareMathOperator*{\st}{s.t.}
\def\prox{{\rm Prox}}
\def\sign{{\rm sign}}
\def\ext{{\rm ext}}
\def\tol{{\rm tol}}

\def\A{{\mathcal A}}
\def\B{{\mathcal B}}

\def\F{{\mathcal F}}

\def\I{{\mathcal I}}
\def\J{{\mathcal J}}
\def\K{{\mathcal K}}
\def\L{{\mathcal L}}

\def\N{{\mathcal N}}
\def\O{{\mathcal O}}
\def\P{{\mathcal P}}
\def\Q{{\mathcal Q}}

\def\T{{\mathcal T}}

\def\V{{\mathcal V}}
\def\W{{\mathcal W}}

\def\Y{{\mathcal Y}}

\usepackage{lipsum}
\usepackage{amsfonts}
\usepackage{graphicx}
\usepackage{epstopdf}
\usepackage{algorithm}%
\usepackage{algorithmicx}%
\ifpdf
\DeclareGraphicsExtensions{.eps,.pdf,.png,.jpg}
\else
\DeclareGraphicsExtensions{.eps}
\fi

\newtheorem{lemma}{Lemma}[section] 
\theoremstyle{thmstyleone}%
\newtheorem{theorem}{Theorem}
\newtheorem{proposition}[theorem]{Proposition}%

\theoremstyle{thmstyletwo}%
\newtheorem{remark}{Remark}%

\theoremstyle{thmstylethree}%

\usepackage{amsopn}

\date{\today}
\begin{document}
\title[Proximal Augmented Lagrangian Method for cMGGMs]{A Semismooth Newton-Based Proximal Augmented Lagrangian Method for Joint Estimation of Multiple Gaussian Graphical Models with Clustered Structure}
\author[1]{\fnm{Han} \sur{Wang}}\email{whmath@126.com}

\author[1]{\fnm{Yue} \sur{Liu}}\email{liuyue@fzu.edu.cn}

\author*[1,2]{\fnm{Yong-Jin} \sur{Liu}}\email{yjliu@fzu.edu.cn}

\affil[1]{\orgdiv{School of Mathematics and Statistics}, \orgname{Fuzhou University}, \orgaddress{\street{No.2 Wulongjiang North Avenue, Fuzhou}, \city{Fuzhou}, \postcode{350108}, \state{Fujian}, \country{P.R. China}}}

\affil[2]{\orgdiv{Center for Applied Mathematics of Fujian Province, School of Mathematics and Statistics}, \orgname{Fuzhou University}, \orgaddress{\street{No.2 Wulongjiang North Avenue}, \city{Fuzhou}, \postcode{350108}, \state{Fujian}, \country{P.R. China}}}

\abstract{
In this paper, we consider a class of convex composite  optimization problems arising from the joint estimation of clustered multiple Gaussian graphical models. The resulting  model   combines a log-determinant loss term with a nonsmooth sparse clustered regularizer,
which encourages both similar sparsity patterns and consistent edge values across different graphs. We first establish a necessary and sufficient condition under which the solution is block diagonal, enabling a large-scale problem to be decomposed into smaller independent subproblems and substantially reducing computational complexity.  To efficiently solve this problem, we develop a  proximal augmented Lagrangian method in which each subproblem is handled by a superlinearly convergent semismooth Newton method. Unlike  widely used first order methods, our approach fully exploits the underlying second order information through the semismooth Newton framework, thereby achieving faster convergence and improved robustness. The efficiency and robustness of the proposed algorithm are demonstrated through comparisons with state-of-the-art methods on both synthetic and real data sets.}
\keywords{
	Gaussian graphical model, clustered structure, screening rule, semismooth Newton method
}
\maketitle

\section{Introduction}

The Gaussian graphical model  (GGM) provides a widely used framework for learning conditional dependence structure from complex high-dimensional data \cite{chan2017gene,colclough2018multi,dondelinger2013non,foti2019statistical,friedman2000using,manning2018probabilistic}. In this framework, the entries of the precision matrix encode  partial correlations between the associated variables \cite{lauritzen1996graphical}. 
Consequently, the structure of Gaussian graph  can be inferred from the sparsity pattern of the precision matrix, which  is typically estimated by  thresholding  small (in absolute value) elements to reveal nonzero entries that represent conditional dependencies between variables  \cite{drton2017structure,drton2004model}.
Estimating sparse precision matrices is commonly formulated as a log-determinant optimization problem with nonsmooth regularization\cite{buhlmann2011statistics,yuan2007model}. Such formulations are statistically meaningful but computationally challenging, especially when the dimension is large and high-accuracy solutions are required.

In many applications,  data sets may exhibit both heterogeneity and homogeneity across  domains or classes.
Directly estimating a single GGM may fail to capture such structured variability, whereas fitting graphical models for distinct classes independently ignores the shared information among the dependency graphs.
The joint graphical model framework \cite{danaher2014joint}  addresses this issue by estimating multiple precision matrices simultaneously while encouraging shared and group-specific structures. 
This joint estimation idea has been further extended to accommodate more complex structural assumptions. For instance,  to handle cases where group labels are latent or unknown,  Hao et al. \cite{hao2018simultaneous} proposed   simultaneous clustering and estimation  procedure.
Bilgrau et al. \cite{bilgrau2020targeted} proposed targeted fused ridge estimator,  which incorporates target matrices as prior information through an $\ell_2$ penalty term to stabilize the estimation process.
These models have improved the recovery of related dependency networks, but they also lead to larger and more tightly coupled nonsmooth optimization problems.

Consider that the data for each group are $p$-variate and share the same set of nodes, but the underlying connection patterns  may be different due to the heterogeneity between groups. 
The data for the $k$-th group can be represented as an $n_k\times p$ matrix $\Delta^k=(\delta ^k_1, \dots, \delta^k_{n_k})^{\top}$, where $\delta^k_i=(\delta^k_{i,1},\dots,\delta^k_{i,p})^{\top}$, $i=1,\dots,n_k$, are $p$-dimensional vectors of observations. 
Suppose that the data in each group are independently drawn from different distributions $\N_p(\mu^k,\Sigma^k)$ for $k=1,\dots,K$. The sample mean and the sample covariance matrice for the $k$-th group  are then given by $\bar{\mu}^k=\frac{1}{n_k}\sum_{i=1}^{n_k}\delta_i^k $, $S^k=\frac{1}{n_k-1}\sum_{i=1}^{n_k}(\delta_i^k-\bar{\mu}^k)(\delta_i^k-\bar{\mu}^k)^{\top}$, $k=1,\dots,K$.
Given observations $\Delta=\{\Delta^1,\dots,\Delta^K\}$, the multiple Gaussian graphical models for estimating the precision matrices $(\Sigma^k)^{-1}$, $k=1,\dots,K$, jointly is the model with the variable $\Theta=(\Theta^1,\dots,\Theta^K)\in\mathbb{S}^p\times\dots\times\mathbb{S}^p$:
\begin{equation}\label{dg:2}
	\min_{\Theta} \:\:\:L(\Theta)+\P(\Theta), 
\end{equation}
where $L(\Theta)=\sum_{k=1}^{K} \left (h(\Theta^k)+\left \langle  S^k,\Theta^k\right \rangle \right )$, with $h(A) = -\log\det A$ for $A\in\mathbb{S}_{++}^p$, and $h(A)=+\infty$ otherwise; $\P(\cdot)$ is a penalty function designed to induce sparsity within each $\Theta^k$ and  capture shared structures among different $\Theta^k$.

In this paper, we focus on a class of joint estimation problems for the clustered multiple Gaussian graphical models (cMGGMs), with a penalty function proposed by Danaher et al. \cite{danaher2014joint}, defined as follows:
\begin{equation}\label{dg:12}
	\P(\Theta)=\mu_1\sum_{j\ne j^{'}}^{}\left( \sum_{k=1}^{K}\left | \Theta _{jj^{'}}^{k} \right |\right)+\mu_2\sum_{j\ne j^{'}}^{}\left( \sum_{k<k^{'}}^{}\left | \Theta _{jj^{'}}^{k}-\Theta _{jj^{'}}^{k^{'}} \right |\right).
\end{equation} 
For $1\le i,j\le p$, we define $\Theta_{[ij]}:=(\Theta_{ij}^1,\dots,\Theta_{ij}^K)^{\top}\in\mathbb{R}^K$ as the column vector obtained by collecting $(i,j)$-th elements from all $K$ matrices $\Theta^k$, $k=1,\dots,K$. It can be clearly seen that 
\begin{equation*} \P(\Theta)=\sum_{i\ne j}^{} \phi(\Theta_{[ij]})\:\: \text{with}\:\: \phi(x)=\mu_1\left \| x \right \|_1+\mu_2\sum_{i<j}^{}|x_i-x_j|, \:\:\:\forall x\in\mathbb{R}^K,
	\end{equation*}
where the function $\phi(\cdot)$ is actually a special sparse clustered Lasso regularizer. It imposes an $\ell_1$ penalty on all  off-diagonal elements of the $K$ precision matrices  and the differences between the corresponding elements of any pair of precision matrices. 
When the tuning parameter $\mu_2$ is sufficiently large, many elements with the same position across the estimated matrices $\Theta^k$, $k=1,\dots,K$ tend to become similar or even identical. Consequently, this clustered graphical Lasso regularizer promotes both sparsity patterns shared across matrices and consistency in the values of corresponding edges among different graphs.

 
Compared with learning a single graphical model, jointly estimating  multiple precision matrices is even more challenging.
The alternating direction method of multipliers (ADMM) \cite{glowinski1975approximation}, a general procedure for solving convex optimization problems, has been widely used for this class of problems.
For example, Danaher et al. \cite{danaher2014joint} applied ADMM to multiple graphical models with clustered and group structures, Gibberd et al. \cite{gibberd2017regularized} focused on multiple graphical models with a group-fused penalty, and Hallac et al. \cite{hallac2017network} employed ADMM for the estimation of time-varying graphical model.
ADMM is a practical first-order method that can efficiently obtain an approximate solution of low to moderate accuracy.
However, ADMM requires computing multiple eigenvalue decompositions at each iteration,  a computationally expensive operation. It hardly utilizes any second-order information, which is generally required to  achieve highly accurate solutions.
To overcome this limitation, several  algorithms have been proposed towards other forms of penalty functions.
For instance, Yang et al. \cite{yang2015fused} employed a proximal Newton-type method 
 \cite{hsieh2011sparse,lee2014proximal} to estimate  the fused graphical Lasso problem.
However, this second-order method in
 \cite{yang2015fused} lacks  guaranteed local linear convergence and incurs high computational cost due to the  quadratic subproblems.
Furthermore, Zhang et al. \cite{zhang2021efficient,zhang2020proximal} proposed a proximal point algorithm for the group and fused  graphical Lasso problems. Each inner subproblem  is solved by the dual-based semismooth Newton method, which effectively exploits second-order information and achieves an arbitrary linear convergence rate.

 
Motivated by the numerical efficiency of semismooth Newton method, we aim  to design an efficient  algorithm  for solving problem (\ref{dg:2}), with economical implementation and fast convergence rate. 
It should be noted that solving the cMGGMs is much more challenging than solving the clustered Lasso problem. 
Specifically, the main difficulties lie in the log-determinant function
and   matrix variables, which make the theoretical analysis and numerical implementation of the algorithms substantially more involved.
The main contributions of this paper can be summarized as follows.
\begin{itemize}
\item[1.]  We derive a surrogate generalized Jacobian of the proximal mapping of (\ref{dg:12}), which involves more complicated manipulations of coordinates for a collection of matrix variables, in contrast to the vector case encountered in the clustered Lasso problem \cite{lin2019efficient}. 
\item[2.] We establish the necessary and sufficient condition for the solution of the cMGGMs to be block diagonal. 
Based on this key property of  (\ref{dg:12}) and several tools from linear programming, we rigorously derive a screening rule that enables efficient estimation of large-scale multiple precision matrices for cMGGMs. The proposed screening rule can be incorporated into any existing algorithm to significantly reduce computational cost.
\item[3.] We design a semismooth Newton-based proximal augmented Lagrangian ({\sc Ssnpal}) method to solve the dual formulation of problem (\ref{dg:2}).
Specifically, we adopt the inexact proximal augmented Lagrangian framework,
which ensures that the unconstrained minimization subproblem at each iteration is strongly convex. Consequently, the semismooth Newton method we employ to solve this subproblem attains locally quadratic convergence.
It should be noted that, for computational efficiency, we warm-start the {\sc Ssnpal} using ADMM so that its fast local convergence can take effect at an early stage.
\end{itemize}

The remaining parts of this paper are as follows. Section \ref{pre} presents some definitions  and preliminary results, including the proximal mapping of (\ref{dg:12}), its generalized Jacobian, the proximal mapping of the log-determinant function, and its derivative.
In section \ref{screen}, we present a
screening rule for the solution of  cMGGMs to be block diagonal for an arbitrary number of graphs.
In section \ref{alg}, we develop the {\sc Ssnpal} method for solving cMGGMs and its convergence properties.  The numerical performance of our proposed algorithm on synthetic and real data sets is evaluated in section \ref{exp}. Section \ref{conclusion} gives the conclusion.

\textbf{Notation.}
Let $\mathbb{S}^p$ denote the space of $p \times p$ real symmetric matrices. We write $\mathbb{S}^p_+$ and $\mathbb{S}^p_{++}$ for the cones of positive semidefinite and positive definite matrices in $\mathbb{S}^p$, respectively.
For any $A$, $B\in\mathbb{S}^p$, we use $A \succeq B$ ( $A \succ B$) to indicate that $A-B \in \mathbb{S}^p_+$ ( $A-B \in \mathbb{S}^p_{++}$). In particular, $A \succeq 0$ ( $A \succ 0$) means that $A$ belongs to $\mathbb{S}^p_+$ ( $\mathbb{S}^p_{++}$).
We define  $\Y \; (\Y_{+}, \Y_{++})$  as   the Cartesian product of $K$ copies of $\mathbb{S}^p \;(\mathbb{S}^p_{+}, \mathbb{S}^p_{++})$.
Let $\mathbb{R}^n$ denote the $n$-dimensional Euclidean space and $\mathbb{R}^{m \times n}$ the set of all $m \times n$ real matrices.
For any matrix $A \in\mathbb{R}^{m\times n}$, $A_{ij}$ denotes the $(i, j)$-th element of $A$.
$\I_n$ represents the $n\times n$ identity matrix and $\I$ denotes the identity matrix or map when the dimension is clear from the context.
Similarly, we use $\mathbf{0}_n$ to denote the zero vector in $\mathbb{R}^n$ and simply write $\mathbf{0}$ when the dimension is clear from the context.
Let $\left \langle \cdot ,\cdot  \right \rangle $  denote the  standard inner product in the space of real matrices and   $\left \| \cdot  \right \| $  the induced   Frobenius norm.
For any $x \in \mathbb{R}^n$, $\left \| x \right \| _1:=\sum_{i=1}^{n} \left | x_i \right | $.
$\Diag(D_1, . . . , D_n)$ denotes the block diagonal matrix whose $i$-th diagonal block is the matrix $D_i$ for $i=1,\dots,n$.
The composition of two functions $f$ and $g$ is denoted by $f \circ g$, i.e., $(f \circ g)(\cdot) = f(g(\cdot))$. The Hadamard  product is written as $\odot$.
For two sequences of numbers $\{ a_n\}$ and $\{ b_n\}$ , we say $a_n=\O (b_n)$ if there exists a  constant $c>0$ such that $| a_n| \leq c| b_n|$ for sufficiently large $n$.

\section{Preliminaries}\label{pre}
In this section, we introduce the definition and some basic properties of the Moreau envelope of a proper, closed convex function, which serve as a foundation for the subsequent theoretical analysis and algorithmic design.
Let $\mathcal{E}$ be a finite-dimensional real Hilbert space and $f: \mathcal{E} \rightarrow \mathbb{R} \cup \{ +\infty \} $ be a given  proper and closed convex function. 
The Moreau envelope  \cite{moreau1965proximite,yosida2012functional} of $f$ is defined by
\begin{equation}\label{eq:11}
	\Psi _{f}(x):=\min_{y\in \mathcal{E}}\left \{ f(y)+\frac{1}{2}\left \| y-x \right \| ^2  \right \},\ \forall x\in \mathcal{E},
\end{equation}
and the proximal mapping associated with $f$,  the unique minimizer of (\ref{eq:11}),  is given by
\begin{equation*}\label{eq:12}
	\prox_{f}(x):=\underset{y\in \mathcal{E}}{\arg\min}\left \{ f(y)+\frac{1}{2}\left \| y-x \right \| ^2  \right \}, \ \forall x\in \mathcal{E}.
\end{equation*} 
It is known that the Moreau envelope of any proper and closed convex function  is  continuously differentiable \cite{lemarechal1997practical,Rockafellar1998VariationalA} with the gradient
\begin{equation*}
	\nabla \Psi _{f}(x)=x-\prox_{f}(x),\:\:\:\forall x\in \mathcal{E}.
\end{equation*}
\subsection{Properties of the  regularizer in (\ref{dg:12})}
In this subsection, we investigate the proximal mapping associated with (\ref{dg:12}) and its generalized Jacobian.
By definition, the proximal mapping of $\P(\cdot)$ is given as follows: for any $X\in\Y$,
\begin{equation}\label{dg:1}
	\begin{split}
		\prox_{\P}(X)&=\underset{\Theta\in\Y}{\arg\min}\left\{\P(\Theta)+\frac{1}{2}\left \| \Theta-X \right \|^2 \right\}\\[2mm]
		&=\underset{\Theta\in\Y}{\arg\min}\left\{\sum_{i\ne j}^{} \left\{\phi(\Theta_{[ij]})+\frac{1}{2}\left \| \Theta_{[ij]}-X_{[ij]} \right \|^2 \right\}+\frac{1}{2}\sum_{i=1}^{p}\left \| \Theta_{[ii]}-X_{[ii]} \right \|^2\right\}.
	\end{split}
\end{equation}
It can be readily observed that the problem (\ref{dg:1}) is block separable with respect to each vector $\Theta_{[ij]}\in\mathbb{R}^K$. Therefore, for any $i,j=1,\dots,p$, the vector $(\prox_{\P}(X))_{[ij]}$, consisting of all entries of $(\prox_{\P}(X))$ in the $(i,j)$-th position, admits the following explicit form
\begin{equation*}
	(\prox_{\P}(X))_{[ij]}=\begin{cases}
		\prox_{\phi}(X_{[ij]}),& \text{ if } i\ne j, \\[2mm]
		X_{[ii]},& \text{ if } i=j.
	\end{cases}
\end{equation*}
Based on the above analysis, the surrogate generalized Jacobian of $\prox_{\P}(\cdot)$ at $X$, denoted by $\hat{\partial}\prox_{\P}(X):\Y\rightrightarrows\Y$, can be expressed as: for any $Y\in\Y$,
	\begin{equation}\label{dg:26}
			\begin{aligned}
				\hat{\partial}\prox_{\P}(X)[Y]
				= \Big\{
				\W[Y] \,\Bigm|\,
				&(\W[Y])_{[ij]} =
				\begin{cases}
					M^{(ij)} Y_{[ij]}, & i<j,\\[2mm]
					Y_{[ii]}, & i=j,\\[2mm]
					M^{(ji)} Y_{[ji]}, & i>j,
				\end{cases}
				 \\[2mm]
				&   M^{(ij)} \in \hat{\partial}\prox_{\phi}(X_{[ij]}), \:i,j=1,\dots,p  
				\Big\}.
			\end{aligned}
		\end{equation}
The  above results show  that computing $\prox_{\P}(\cdot)$ and its generalized Jacobian reduces to performing $p(p-1)/2$ independent computations of $\prox_{\phi}(\cdot)$ and its generalized Jacobian, which can be executed in parallel. The explicit expressions of the proximal mapping $\prox_{\phi}(\cdot)$ and its generalized Jacobian $\hat{\partial}\prox_{\phi}(\cdot)$ have been derived in \cite{lin2019efficient}. For completeness, we briefly recall the main results concerning the computation of $\prox_{\phi}$ and $\hat{\partial}\prox_{\phi}(\cdot)$, as presented in  \cite[Proposition 2.4, Proposition 2.9]{lin2019efficient}.
\begin{proposition}
	Let $y\in\mathbb{R}^n$ be given. Then there exists a permutation matrix $P_y\in \mathbb{R}^{n\times n}$ such that $\tilde{y}=P_yy$ satisfies $\tilde{y}_1\ge\tilde{y}_2\ge\dots\ge\tilde{y}_n$. It follows that 
\[ \prox_{\phi}(y)=\prox_{\mu_1 \left\|\cdot\right\|_1}(G_{\mu_2}(y))=\sign(G_{\mu_2}(y))\circ\max(|G_{\mu_2}(y)|-\mu_1,0), \] 
where $G_{\mu_2}(y)=P_y^{\top}\Pi_{F}(P_yy-\mu_2 \omega)$, $\omega\in\mathbb{R}^n$ is the vector with the entries $\omega_k=n-2k+1$, $k=1,\dots,n$, and $F=\left \{ x\in \mathbb{R}^{n}|Ex\ge 0  \right \}$  with $Ex=(x_1-x_2,\dots,x_{n-1}-x_n)^{\top}\in \mathbb{R}^{n-1}$.
\end{proposition}

\begin{proposition}
	For any $y\in\mathbb{R}^n$, the generalized Jacobian of $\prox_{\phi}(\cdot)$ at $y$ is given by
	\begin{equation*}
		\hat{\partial}\prox_{\phi}(y)= \left \{ H\in \mathbb{S}^n\:|\:H=RN, R\in \partial_B\prox_{\mu_1\left \| \cdot  \right \|_1 } (G_{\mu _2}(y)),N\in \J_{G_{\mu2}}(y)\right \},	
	\end{equation*}
	where $\partial_B\prox_{\mu_1\left \| \cdot  \right \|_1 }(\cdot)$ denotes  the B-subdifferential of $\prox_{\mu_1\left \| \cdot  \right \|_1 }(\cdot) $, $\J_{G_{\mu2}}(\cdot)$ denotes the generalized Jacobian of $G_{\mu_2}(\cdot)$ \cite{lin2019efficient}.
\end{proposition}

Based on \cite[Theorem 2.10]{lin2019efficient}, the following result clarifies the rationale for treating $\hat{\partial}\prox_{\P}(X)$ in (\ref{dg:26}) as the surrogate  generalized Jacobian of $\prox_{\P}(\cdot)$ at $X$. The proof follows directly from the cited theorem and is therefore omitted.
\begin{theorem}
	The multifunction $\hat{\partial}\prox_{\P}(X)$ defined in (\ref{dg:26}) is  nonempty, compact, and upper semicontinuous. For any $X \in \Y$, any element in the set $\hat{\partial}\prox_{\P}(X)$ is both self-adjoint and positive semidefinite. Furthermore, there exists a neighborhood $\N (X)$ of $X$ such that for all $Y \in \N (X)$, 
     	\begin{equation*}
     		\prox_{\P}(Y)-\prox_{\P}(X)-\W[Y-X]=0,\:\: \forall \W \in \hat{\partial}\prox_{\P}(X).
     	\end{equation*}
\end{theorem}

\subsection{Properties of the log-determinant function}
In this subsection, we present several properties of the proximal mapping associated with the following log-determinant function  and its derivative,  which are drawn mainly  from  \cite{wang2010solving,yang2013proximal}. 
To describe the properties of $\prox_h(\cdot)$, we introduce two scalar functions associated with a given $\beta>0$:
\begin{equation*} 
\phi_{\beta}^+(d):=(\sqrt{d^2+4\beta}+d)/2,\:\:\:\phi_{\beta}^-(d):=(\sqrt{d^2+4\beta}-d)/2,\:\:\forall d\in\mathbb{R}.
\end{equation*}
Moreover, for any symmetric matrix $A \in \mathbb{S}^p$ with eigenvalue decomposition $A = Q \allowbreak \Diag(d_1,\allowbreak \dots,\allowbreak d_p) Q^\top$, the matrix counterparts of these two scalar functions are defined by
\begin{equation}\label{dg:27} \phi_{\beta}^+(A):=Q\Diag(\phi_{\beta}^+(d_1),\dots,\phi_{\beta}^+(d_p))Q^{\top},\:\:\:\phi_{\beta}^-(A):=Q\Diag(\phi_{\beta}^-(d_1),\dots,\phi_{\beta}^-(d_p))Q^{\top}. 
\end{equation}
It can be readily verified that the functions $\phi_{\beta}^+(\cdot)$ and $\phi_{\beta}^-(\cdot)$ are well-defined and yield positive definite matrices for any $A \in \mathbb{S}^p$. Based on these functions, the subsequent proposition provides explicit expressions for the proximal mapping 
 \cite[Proposition 2.3]{yang2013proximal} of the log-determinant function $h(\cdot)$ and  its  derivative \cite[Lemma 2.1(b)]{wang2010solving}.
\begin{proposition}
	Let $\beta >0$ be a given scalar and $A \in \mathbb{S}^p$. Define the functions $\phi_\beta^+(A)$ and $\phi_\beta^-(A)$ as in (\ref{dg:27}). Then the following statements hold:
	\begin{itemize}
		\item [(i)] The proximal mapping and Moreau envelope of $h(\cdot)$ at $A$ with parameter $\beta$ are given by $\:
			\prox_{\beta h}(A)\allowbreak	  =\phi_{\beta}^+(A),
				\Psi_{\beta h}(A) =-\beta\log\det(\phi_{\beta}^+(A))	+\frac{1}{2}\left \|\phi_{\beta}^-(A)\right \|^2.$
		\item [(ii)] The mapping $\phi_\beta^+: \mathbb{S}^p \to \mathbb{S}^p$ is continuously differentiable. Moreover, if $A$ has the eigenvalue decomposition $A = Q \Diag(d_1, \dots, d_p) Q^\top$, then the directional derivative of $\phi_\beta^+(\cdot)$ at $A$ along any $B \in \mathbb{S}^p$ is given by
		\begin{equation*} (\phi_{\beta}^+)'(A)[B]=Q(\Gamma \odot (Q^{\top}BQ))Q^{\top},
			\end{equation*}
		where the entries of 
  $\Gamma\in\mathbb{S}^p$ are
		\begin{equation*} \Gamma_{ij}=\frac{\phi_{\beta}^+(d_i)+\phi_{\beta}^+(d_j)}{(d_i^2+4\beta)^{1/2}+(d_j^2+4\beta)^{1/2}},\:\:i,j=1,\dots,p.
			\end{equation*}
	\end{itemize}
\end{proposition}

\section{The screening rule for the cMGGMs}\label{screen}

In this section, we present a theorem that provides substantial computational improvements to the {\sc Ssnpal} algorithm discussed in Section \ref{alg}. 
This theorem only needs to be examined for the empirical covariance matrices $S^1,\dots, S^K$,  and  the given parameters. According to this screening strategy, we can simply perform the {\sc Ssnpal} algorithm on the features within each block separately, to obtain exactly the same solution that would have been obtained  by  applying the algorithm to all $p$ features. It results in huge speed improvements, as it eliminates the need  to compute the eigenvalue decomposition of a $p\times p$ matrix.

Suppose that  $p$ features are divided into  $T$ non-overlapping sets $C_1,\dots,C_T$, with $C_t\cap C_{t'}=\emptyset $ for all $t\ne t'$ and $ \bigcup_{t=1}^{T} C_t=\left \{ 1,\dots,p \right \} $. We call the solution $\hat{\Theta}$ of the problem (\ref{dg:2}) a block diagonal form and contains $T$ known blocks, each of which is composed of features from the sets $C_t$, $t=1,\dots,T$, if there exists a permutation matrix $U\in\mathbb{R}^{p\times p}$ such that each estimated precision matrix takes the form of 
\begin{equation}\label{dg:13}
	\hat{\Theta}^{k}=U\begin{pmatrix}
		\hat{\Theta}^{k}_1&  & \\
		& \ddots  & \\
		&  &\hat{\Theta}^{k}_T
	\end{pmatrix} U^{\top },\:\: k=1,\dots,K.
\end{equation}
For   convenience in the subsequent analysis, we assume  $U=\I$ in this paper.

The following decomposition result for problem (\ref{dg:2}) is trivial and we omit the proof.
\begin{theorem}\label{Th:1}
	Suppose that the solution $\hat{\Theta}$ of problem (\ref{dg:2}) is block diagonal with $T$ known $C_t$, $t=1,\dots,T$, i.e., each estimated precision matrix takes the form specified in (\ref{dg:13}) with $U=\I$.  Then, the submatrices $\hat{\Theta}_t=(\hat{\Theta}_t^1,\dots,\hat{\Theta}_t^K)$,  $t=1,\dots,T$,  can be obtained by solving the problem (\ref{dg:2}) separately on  the corresponding set of features only.
	
\end{theorem}

From the above theorem, it can be seen that when  the large-scale  cMGGMs  admits a block diagonal solution, the original problem can be decomposed into a group of smaller scale  problems. Obviously, this decomposition can significantly reduce the computational cost. It raises a key question: How to efficiently identify the block diagonal structure of the cMGGMs solution  before solving the problem? 
To explore this issue, Danaher et al. 
 \cite{danaher2014joint} provided the  necessary and sufficient condition for the problem (\ref{dg:2})   to have a block diagonal solution, but only for the special case of $K=2$. 
Subsequently, Yang et al.  \cite{yang2015fused} derived an equivalent condition to determine whether the solution to the fused graphical Lasso  possesses a block diagonal structure when $K>2$.
In addition, Yang et al. 
 \cite{yang2015fused} proposed a screening strategy for general fused regularization. However, the theoretical validity of this strategy was not established in a systematic manner, which limits its applicability to specific regularization models. 
In the subsequent content of this section, we will adapt the ideas in 
 \cite{yang2015fused} to derive the necessary and sufficient condition for the solution of the cMGGMs to be block diagonal for $K>2$ graphs.

First of all, we establish several technical lemmas as follows.

\begin{lemma}\label{le:2}
	Let $I\subseteq \{1,\dots,n\}$ and $J\subseteq\{1,\dots,\frac{n(n-1)}{2}\}$ be two arbitrary index sets , 
	with  $\bar{I}$ and $\bar{J}$ being  the respective complements in $\{1,\dots,n\} $  and $\{1,\dots,\frac{n(n-1)}{2}\}$. 
	Consider a matrix $C$ such that 
	\begin{equation*} Cy=[y_1-y_2; \dots; y_1-y_n; y_2-y_3; \dots; y_2-y_n; \dots ;y_{n-1}-y_n]\in\mathbb{R}^{\frac{n(n-1)}{2}}.
		\end{equation*}
	Define
	\begin{equation}\label{dg:15}
		P_{I,J}=\left \{ y\in \mathbb{R}^n \:|\:y_I\ge \mathbf{0}, y_{\bar{I} }\le \mathbf{0}, C_Jy\ge\mathbf{0},  C_{\bar{J} }y\le \mathbf{0} \right \} ,
	\end{equation}	
	where $C_J$ is the matrix consisting of the rows of $C$ indexed by $J$.
	Then, the following statements hold:
	\begin{itemize}
		\item [(a)] $P_{I,J}$ is either equal to $ \{\mathbf{0}\}$ or unbounded, with the origin being its unique extreme point in both cases.
		\item[(b)] Assume that $P_{I,J}$ is unbounded. Then, the set of all extreme rays of $P_{I,J}$ denoted by $\text{ext
		}(P_{I,J})$ is nonempty and satisfies  
		\begin{equation*} \cup \left \{ \ext(P_{I,J})\: |\: I\subseteq \{1,\dots,n\},  J\subseteq\{1,\dots,\frac{n(n-1)}{2}\} \right \} =Q,
			\end{equation*}
		where
		\begin{equation*}\label{dg:16}
			Q:=\{\alpha d \: |\: d_i=1 \:\text{or}\:  0, i=1,\dots,n, \alpha \ne 0 \}\setminus \{ \mathbf{0} \}.
		\end{equation*}
	\end{itemize}
\end{lemma}
\begin{proof} We shall prove the assertions (a) and (b) in turn.
	
	\begin{itemize}
		\item  [(a)] From the definition of $P_{I,J}$, it is immediate that  either  $P_{I,J}=\{\mathbf{0}\}$ or $P_{I,J}$ is unbounded, and in both cases,  $\mathbf{0}$ is an extreme point of $P_{I,J}$. Thus, we only need to prove the uniqueness of the extreme point. Assume for the sake of contradiction that there exists an extreme point $y\ne \mathbf{0}$ in $P_{I,J}$.  Suppose that, among the inequalities $C_Jy\ge\mathbf{0}$ and $C_{\bar{J} }y\le\mathbf{0} $, there are $n_1 \,(1\le n_1 \le n-1)$ linearly independent active inequalities at $y$, and the corresponding components of $y$ are nonzero. 
		This means that at least $n_1+1$ components of $y$ satisfy these  linearly independent active inequalities.
		Since $y$ has at most $n-n_1-1$ zero entries with respect to $P_{I,J}$, it follows that the total number of linearly independent active inequalities at $y$ cannot exceed $n-1$. This result contradicts   the assumption, and hence $\mathbf{0}$ is the unique extreme point of $P_{I,J}$.
		\item[(b)] First, we show that $  \cup \left \{ \text{ext}(P_{I,J})\: |\: I\subseteq \{1,\dots,n\},  J\subseteq\{1,\dots,\frac{n(n-1)}{2}\} \right \}\subseteq Q$. Given   two arbitrary index sets $I\subseteq \{1,\dots,n\}$ and $J\subseteq\{1,\dots,\frac{n(n-1)}{2}\}$, let $z\in \text{ext}(P_{I,J})$ be arbitrarily chosen. Without loss of generality,  assume that $z$ has exactly $n_2\,(0\le n_2\le n-1)$ zero components. It follows that $z$ contains $n-n_2$ nonzero components  and that there are $n-n_2-1$ linearly independent active inequalities from $C_Jz\ge\mathbf{0}$ and $C_{\bar{J} }z\le\mathbf{0} $ at $z$, which are linearly independent of those arising from $z_I\ge \mathbf{\mathbf{0}}$ and $z_{\bar{I} }\le \mathbf{0}$.
		Therefore, the $n-n_2$ nonzero components of $z$ are equal, which implies $z\in Q$.
In contrast, $\forall z\in Q$, let $\K=\{i\:|\: z_i=\alpha\ne 0,\:i=1,\dots,n \}$ and $\bar{\K}$ be the complement of $\K$ with respect to $\{1,\dots,n\}$. If $\alpha>0$, it easily follows that $z\in\text{ext}(P_{I,J})$ for $I=\{1,\dots,n\}$ and $J=\{j\:|\: j=\frac{(i_0-1)(2n-i_0)}{2}+t,\:\:\forall i_0\in \K,\:t\in \{1,\dots,n-i_0\}\}$.    Analogously, if $\alpha<0$, it can be shown that $d\in\text{ext}(P_{I,J})$ with $I=\emptyset$ and $J$ being the complement of $\bar{J}=\{j\:|\: j=\frac{(i_0-1)(2n-i_0)}{2}+t,\:\:\forall i_0\in \K,\:t\in \{1,\dots,n-i_0\}\}$. This means that $d\in \cup \left \{ \text{ext}(P_{I,J})\: |\: I\subseteq \{1,\dots,n\},  J\subseteq\{1,\dots,\frac{n(n-1)}{2}\} \right \}$. In summary, we conclude that the equivalence is valid.                                                                         
 \end{itemize}

The proof is complete.
\end{proof}

Lemma \ref{le:2} establishes the fundamental properties of the set $P_{I,J}$, which form the basis  for deriving equivalence results in the subsequent analysis.

\begin{lemma}\label{le:3}
	Given $x\in\mathbb{R}^n$, $\mu_1$, $\mu_2> 0$,  define
	\begin{equation*} 
		f(y):=\left \langle  x,y\right \rangle -\mu_1\left \| y \right \| _1-\mu_2\left \| Cy \right \|_1 .
		\end{equation*}
	Then, $f(y)\le 0$ for all $y\in \mathbb{R}^n$ if and only if $x$ satisfies the following inequalities:
	\begin{equation*} \left | \sum_{i\in \A }^{} x_i \right | \le|\A |\mu_1+|\A |(n-|\A |) \mu_2
	\end{equation*}
	for all nonempty subsets $ \A \subseteq \left \{ 1,\dots,n \right \} $.
\end{lemma}
\begin{proof}
	From the definition of $P_{I,J}$ in (\ref{dg:15}),  it can be shown that 
	\begin{equation*} \mathbb{R}^n=\cup \left \{ P_{I,J}\: |\: I\subseteq \{1,\dots,n\},  J\subseteq\{1,\dots,\frac{n(n-1)}{2}\} \right \}.
		\end{equation*}   Therefore, $\forall y\in\mathbb{R}^n$, $f(y)\le 0$ if and only if $ f(y)\le 0$ for all $y\in P_{I,J}$ and every subset $I\subseteq \{1,\dots,n\}$,  $J\subseteq\{1,\dots, \frac{n(n-1)}{2}\}$. 
	Next, we consider the following two cases:
	\begin{itemize}
		\item[(i)] $P_{I,J}$ is bounded. In this case, it follows from  Lemma \ref{le:2} (a) that $P_{I,J}=\{\mathbf{0}\}$ and $f(y)=0$ for $y\in P_{I,J}$. 
		\item[(ii)] $P_{I,J}$ is unbounded. It is easily verified that $P_{I,J}$ is the finitely generated cone generated by $\text{ext}(P_{I,J})$.  Consequently, the inequality $ f(y)\le 0$ holds for all $y\in P_{I,J}$ if and only if   it holds for all $y\in \text{ext}(P_{I,J})$. Together with Lemma \ref{le:2} (b), we see that $ f(y)\le 0$ for all $y\in \text{ext}(P_{I,J})$ is equivalent to   $ f(y)\le 0$ for all $y\in Q$. By the definitions of $Q$ and $f$, it then follows that the desired equivalence holds.
	\end{itemize}

    The proof is complete.
\end{proof}

The function $f(\cdot )$ is instrumental in deriving the necessary and sufficient condition for identifying the block structured solution. Next, we are ready to present a key equivalence relation that will support the analysis for identifying the block structure.

\begin{lemma}\label{le:1}
	Let $x\in\mathbb{R}^n$, $\mu_1$, $\mu_2> 0$ be given. Then the following linear system
	\begin{equation}\label{dg:23}
		\left\{\begin{matrix}
			x+\mu_1\gamma +\mu_2C^{\top }\nu =0,\:\:\:\:\:\:\:\:\:\:\:\:\:\:\:\:\:\:\\[2mm]
			-1\le \gamma _i\le 1,\:\: i=1,\dots,n, \:\:\:\:\:\:\:\:\:\\[2mm]
			-1\le \nu  _i\le 1,\:\: i=1,\dots,\frac{n(n-1)}{2}  
		\end{matrix}\right.
	\end{equation}
	has a solution $(\gamma,\nu)$ if and only if $(x,\mu_1,\mu_2)$ satisfies
	\begin{equation*}
	 \left | \sum_{i\in \A }^{} x_i \right | \le|\A |\mu_1+|\A |(n-|\A |) \mu_2
	\end{equation*}
	for all nonempty subsets $\A \subseteq \left \{ 1,\dots,n \right \} $.
\end{lemma}
\begin{proof}
	The solvability of the linear system (\ref{dg:23}) is equivalent to  the existence of an optimal solution to the following linear program:
	\begin{equation}\label{dg:24}
		\min_{\gamma ,\nu } \left \{ 0\:|\: (\gamma ,\nu)\text{ satisfies } (\ref{dg:23}) \right \}.
	\end{equation}	
	  Then, the Lagrangian dual of (\ref{dg:24}) can be simplified to the following form:
	\begin{equation}\label{dg:25}
		\max_{y} f(y):=\left \langle  x,y\right \rangle -\mu_1\left \| y \right \| _1-\mu_2\left \| Cy \right \| _1.
	\end{equation}
	From the Strong Duality Theorem \cite[Proposition 6.4.4]{bertsekas2003convex}, 
	(\ref{dg:24}) admits an optimal solution if and only if its dual  (\ref{dg:25}) achieves an optimal value of zero, that is, $f(y)\le 0$ for all $y$. Therefore, the result of this lemma follows directly from Lemma \ref{le:3}.
\end{proof}

We are now in a position to establish the necessary and sufficient condition under which the solution to the cMGGMs exhibits a block diagonal structure.

\begin{theorem}\label{Th:2}
	Consider the cMGGMs. Then, it has a block diagonal solution $\hat{\Theta}^k$, $k=1,\dots,K$, with $T$ known blocks $C_t$, $t=1,\dots,T$, if and only if $S^k$, $k=1,\dots,K$, $\mu_1$, and $\mu_2$ satisfy the following conditions:
	\begin{equation}\label{dg:20}
		\left | \sum_{k\in \A }^{} S_{ij}^k \right | \le|\A |\mu_1+|\A |(K-|\A |) \mu_2
	\end{equation} 
	for all nonempty subsets $\A \subseteq \left \{ 1,\dots,K \right \} $, $i\in C_t$, $j\in C_{t'}$, $t\ne t'$.
\end{theorem}
\begin{proof}
	For simplicity, let $\hat{\Upsilon }^k=(\hat{\Theta}^k)^{-1}$ for $k=1,\dots,K$.  $\hat{\Theta} =(\hat{\Theta}^1,\cdots , \hat{\Theta}^K)$ is the optimal solution to the problem (\ref{dg:2}) if and only if it satisfies the following system:
	\begin{equation}\label{dg:17}
		\left\{\begin{matrix}
			-\hat{\Upsilon }_{[ii]}+S_{[ii]}=0, \:\:\:\:\:\:\:\:\:\:\:\:\:\:\:\:\:\:\:\:\:\:\:\:\:\:\:\:\:\:\:\:\:\:\:\:\:\:\:\:\:\:\\[2mm]
			-\hat{\Upsilon }_{[ij]}+S_{[ij]}+\mu_1\gamma_{[ij]}+\mu_2C^{\top}\nu_{[ij]}=0 
		\end{matrix}\right.
	\end{equation}
	for all $i \ne j$, where $\gamma_{[ij]} \in \partial \left \|   \Theta_{[ij]}\right \|_1$ and $C^\top \nu_{[ij]} \in \partial \left \|   C\Theta_{[ij]}\right \|_1$ at $\Theta_{[ij]} = \hat{\Theta}_{[ij]}$.

	First,  suppose that $\hat{\Theta} =(\hat{\Theta}^1,\cdots , \hat{\Theta}^K)$ is a block diagonal optimal solution to problem (\ref{dg:2}) with $T$ known blocks $C_t$, $t=1,\dots,T$. Then, $\hat{\Upsilon } =(\hat{\Upsilon }^1,\cdots , \hat{\Upsilon }^K)$ shares the same block diagonal structure, implying $\hat{\Upsilon }^k_{ij}=\hat{\Theta}^k_{ij}=0$ for all $i\in C_t$, $j\in C_{t'}$ with $t\ne t'$. Together with (\ref{dg:17}), it implies that for each $i\in C_t$, $j\in C_{t'}$, $t\ne t'$, there exist $\gamma_{[ij]}\in\mathbb{R}^K$, $\nu_{[ij]}\in\mathbb{R}^{\frac{K(K-1)}{2}}$ such that
	\begin{equation}\label{dg:18}
		\left\{\begin{matrix}
			S_{[ij]}+\mu_1\gamma _{[ij]}+\mu_2C^{\top}\nu _{[ij]}=0,\:\:\:\:\:\:\:\:\:\:\:\:\:\:\:\:\:\:\:\\[2mm]
			-1\le (\gamma _{[ij]})_k\le1,\:\: k=1,\dots,K,  \:\:\:\:\:\:\:\:\:\:\:\:\:\\[2mm]
			-1\le (\nu _{[ij]})_k\le1,\:\: k=1,\dots,\frac{K(K-1)}{2}.
		\end{matrix}\right.
	\end{equation}
	By combining (\ref{dg:18}) with Lemma \ref{le:1}, it follows that the desired result (\ref{dg:20}) holds.

	Conversely, suppose that  (\ref{dg:20}) holds for all nonempty subsets $ \A \subseteq \left \{ 1,\dots,K \right \} $ and all $i\in C_t$, $j\in C_{t'}$ with $t\ne t'$. Then, by Lemma \ref{le:1},  there exist $\gamma_{[ij]}\in\mathbb{R}^K$ and  $\nu_{[ij]}\in\mathbb{R}^{\frac{K(K-1)}{2}}$ satisfying  (\ref{dg:18}) for each corresponding $(i,j)$. Since $\hat{\Theta}$ is the optimal  solution to problem (\ref{dg:2}), the first order optimality condition  implies that (\ref{dg:17}) holds  for all $i\ne j$.  Then $\hat{\Upsilon }^k$ is block diagonal and hence $\hat{\Theta}^k = (\hat{\Upsilon }^k)^{-1}$ is also block diagonal for each $k = 1,\dots,K$, with the same block partition ${C_1, \dots, C_T}$.  The conclusion thus holds.
\end{proof}

Theorem \ref{Th:2} provides a useful screening rule based on condition (\ref{dg:20}) to identify block diagonal structure  in the solution to cMGGMs. By Theorems \ref{Th:1} and \ref{Th:2}, the original large scale problem can be decomposed into smaller subproblems, each restricted to a connected component, thereby reducing the overall computational cost.
However, due to the form of inequalities (\ref{dg:20}), the effectiveness of this screening rule is influenced by the parameters $\mu_1$ and $\mu_2$. When $\mu_1$ and $\mu_2$ are small, only a limited number of  block structures can be identified. 
Therefore, developing more broadly applicable and efficient screening strategies remains an important direction for future research.

\section{An Algorithm for the cMGGMs}\label{alg}
In this section, we develop an efficient algorithm with superlinear convergence for cMGGMs by applying a proximal augmented Lagrangian framework to the dual formulation. The added proximal term ensures the well-posedness of the subproblems, thereby facilitating efficient semismooth Newton updates.

Problem (\ref{dg:2})  can be equivalently written in the constrained form
\begin{equation}\label{dg:3}
	\begin{aligned}
		\underset{\Theta,\Omega}{\min }& \:\:\:L(\Theta) +\P(\Omega) \\
		\text{s.t.}	&\:\:\:\Theta=\Omega,
	\end{aligned}
\end{equation}	
whose corresponding Lagrangian dual problem is given by
\begin{equation}\label{dg:7}
	\underset{(X^{k})_{k=1}^K} {\max} -\left \{ W(X):= L^*(-X)+\P^*(X)\right \},
\end{equation}
where $L^*(\cdot)$ and $\P^*(\cdot)$ are the conjugate functions of $L(\cdot)$ and $\P(\cdot)$, respectively.

\subsection{A Proximal Augmented Lagrangian Method for the Dual Problem}\label{sec:algorithm}
To develop an augmented Lagrangian type method for solving (\ref{dg:7}), we
adopt the framework established by Rockafellar and Wets 
 \cite[Definition 11.45, Example 11.57]{Rockafellar1998VariationalA}. Specifically, 
we rewrite the dual problem (\ref{dg:7}) as the minimization of the function $W(X)=V(X;0,0)$, where $V:\Y\times\Y\times\Y\to (-\infty ,+\infty ]$ is defined by
\begin{equation*} V(X;Z_1,Z_2):= L^*(-X+Z_1)+\P^*(X+Z_2).
	\end{equation*}
The associated  Lagrangian function $\L:\Y\times\Y\times\Y\to(-\infty ,+\infty ]$  is then given by
\begin{equation}\label{dg:4}
	\begin{aligned}
			\L(X;\Theta,\Omega):&=\underset{(Z_1,Z_2)\in \Y\times \Y}{\inf } \left \{ V(X;Z_1,Z_2)-\left \langle \Theta,Z_1 \right \rangle -\left \langle \Omega,Z_2 \right \rangle  \right \}\\
			&=-L(\Theta)-\P(\Omega)-\left \langle \Theta,X \right \rangle+\left \langle\Omega,X \right \rangle. 
		\end{aligned}
\end{equation}
For any given positive scalar $\sigma$, the corresponding augmented Lagrangian function is then given by
\begin{equation*}
	\begin{aligned}
		&\L_{\sigma}(X;\Theta,\Omega)\\[2mm]
		=&\underset{(Z_1,Z_2)\in \Y\times \Y}{\sup } \left \{ \L(X;Z_1,Z_2)-\frac{1}{2\sigma }\left \| Z_1-\Theta \right \|^2-\frac{1}{2\sigma }\left \| Z_2-\Omega \right \|^2   \right \}\\[2mm]
		=&-\sum_{k=1}^{K} \frac{1}{\sigma} \Psi _{\sigma h}(\Theta^k-\sigma (X^k+S^k))+\sum_{k=1}^{K}\left ( \frac{1}{2\sigma }\left \| \Theta^k -\sigma (X^k+S^k)\right \|^2- \frac{1}{2\sigma }\left \| \Theta^k \right \|^2  \right ) \\[2mm]
		&- \frac{1}{\sigma} \Psi _{\sigma \P}(\Omega+\sigma X)+ \frac{1}{2\sigma }\left \| \Omega+\sigma X\right \|^2- \frac{1}{2\sigma }\left \| \Omega \right \|^2. 
	\end{aligned}
\end{equation*}

A  direct application of the classical augmented Lagrangian method (ALM) to problem (\ref{dg:7}) yields an $X$-subproblem, expressed as
\begin{equation*} \underset{X}{ \min }\:\:\L_{\sigma}(X;\Theta,\Omega),
	\end{equation*}
where the objective function is continuously differentiable due to the smoothness of $\Psi _{\sigma h}(\cdot)$ and $\Psi _{\sigma \P}(\cdot)$. The first order optimality condition of  this subproblem can be written as
\begin{equation*}
	\begin{aligned}
		\nabla_X \L_{\sigma}(X;\Theta,\Omega)=&-\left ( \prox_{\sigma h} \left(\Theta^1-\sigma(X^1+S^1)\right),\dots,\prox_{\sigma h} \left(\Theta^K-\sigma(X^K+S^K)\right)\right )\\[2mm]
		&+\prox_{\sigma \P}\left ( \Omega+\sigma X \right ).
	\end{aligned}
\end{equation*}
However, Newton-type methods cannot be applied directly to this nonlinear system as the associated generalized Hessian may be singular. To enhance robustness while maintaining the convergence properties of the ALM framework, a proximal  term $\frac{\tau_t}{2\sigma_t} \left \|  X-X^{(t)} \right \|^2$ is introduced into the subproblem,  following the proximal ALM strategy proposed in  \cite{rockafellar1976augmented}. 
This modification renders this subproblem strongly convex, which   avoids singularity in the generalized Hessian.
The resulting procedure constitutes the inexact proximal ALM, as detailed in Algorithm \ref{alg:dg1}.


\begin{algorithm}[H]
	\caption{A Proximal Augmented Lagrangian Method for the Dual Problem  (\ref{dg:7})}\label{alg:dg1}
	\hspace*{0.01in} \raggedright {\bf Input:} $\sigma_0, \tau_0 >0$, $(X^{(0)},\Theta^{(0)}, \Omega^{(0)})\in \Y_{++}\times\Y_{++}\times\Y_{++}$. Set $t=0$. \\
	\begin{algorithmic}[1]
		\State  Approximately compute
		\begin{equation}\label{dg:5}
			X^{(t+1)}\approx \underset{X}{ \arg \min } \left \{ \Gamma_t(X):=\L_{\sigma_t}(X;\Theta^{(t)},\Omega^{(t)})+\frac{\tau_t}{2\sigma_t} \left \|  X-X^{(t)} \right \|^2\right \} 
		\end{equation}
		to satisfy the conditions (A) and (B) below.
		
		\State Update the multipliers
		\begin{subequations}
			\begin{numcases}{} 
				\Omega^{(t+1)}=\prox_{\sigma _t\P}(\Omega^{(t)}+\sigma_tX^{(t+1)}),\notag\\ 
				(\Theta^k)^{(t+1)}=\prox_{\sigma _th}((\Theta^k)^{(t)}-\sigma_t((X^k)^{(t+1)}+S^k)),\:\: k=1,\dots,K.	\notag
			\end{numcases}
		\end{subequations}
		
		\State Update $\sigma_{t+1}\uparrow \sigma_{\infty}\le\infty$, $\tau_{t+1}\downarrow \tau_{\infty}>0$, and go to Step 1.
	\end{algorithmic}
\end{algorithm}

\begin{remark}
Based on our preliminary experiments, the initial parameter is set as $\sigma_0 = \max\{0.02,\, \min\{1,\,  \mu_1,\, 1/\left \| S \right \| \}\}$ and $\tau_0=1$. 
The parameters $\sigma_t$ and $\tau_t$ are then updated adaptively according to the following rules:
\[
\sigma_{t+1} =
\begin{cases}
\zeta\, \sigma_t, & \text{if } \dfrac{\chi_t}{\chi_{t-1}} > 0.6,\\[1em]
\sigma_t, & \text{otherwise},
\end{cases}
\quad
\text{and} \quad \tau_t = \sigma_t \cdot \max\{10^{-12},\, 10^{-2}\cdot t^{-2.5}\},
\]
where $\zeta = 2$ if $\sigma_t < 10^7$ and $\zeta = 1.3$ otherwise, $\chi_t$ denotes the  dual feasibility at the $t$-th iteration.
\end{remark}
 
 We consider the maximal monotone operator $\T$, defined via the subgradient of the convex-concave Lagrangian function in (\ref{dg:4}) by
\[
\begin{aligned}
	\T(X,\Theta,\Omega):
	&= \left \{ (X',\Theta',\Omega') \right. &|&\left. (X',-\Theta',-\Omega')\in \partial \L(X;\Theta,\Omega)\right \} \\[2mm]
	&=\left \{  (X',\Theta',\Omega') \right. & |&\left.X'=-\Theta+\Omega, \:\Omega'\in \partial \P(\Omega)-X,\right.  \\
	& &   &\left.(\Theta')^k=-(\Theta^k)^{-1}+X^k+S^k,\:k=1,\dots,K \right \}.
\end{aligned}
\]
Let $\Q_t=\Diag(\tau_t \I, \I, \I)$ be a block diagonal matrix. Thus, each iteration of Algorithm~\ref{alg:dg1} can  be regarded as the following approximate rule:
\begin{equation*}(X^{(t+1)}, \Theta^{(t+1)}, \Omega^{(t+1)})\approx (\Q_t+\sigma_t\T)^{-1}\Q_t(X^{(t)}, \Theta^{(t)}, \Omega^{(t)}),
	\end{equation*}
which places it within the framework of the preconditional proximal point algorithm. The convergence properties can therefore be analyzed using the results in \cite[Theorem 2.3 and Theorem 2.5]{li2020asymptotically}.

Since the function $h(\cdot)$ is strictly convex, the associated Karush--Kuhn--Tucker (KKT) system admits a unique solution, denoted by $(\bar{X},\bar{\Theta},\bar{\Omega})$, and thus satisfies $\T^{-1}(0)=\left \{(\bar{X},\bar{\Theta},\bar{\Omega})\right \}$.
Note that the regularization $\P(\cdot)$, defined in (\ref{dg:12}),  is a positive homogeneous function. 
According to \cite[Example 11.4(a)]{Rockafellar1998VariationalA}, its conjugate function $\P^*(\cdot)$ is therefore the indicator function of a nonempty convex polyhedral set.
Furthermore,  it follows  from \cite[Theorem 2.7]{li2018highly}, \cite[Proposition 6]{cui2019r}, and \cite[Lemma 2]{li2020asymptotically} that the inverse operator $\T^{-1}$, evaluated at the origin,  satisfies the following error bound property: 
for given $r>0$, there exists a constant $\kappa>0 $, such that for any $(X,\Theta, \Omega)\in\Y \times\Y\times \Y$ satisfying $\left \| (X,\Theta,\Omega)-(\bar{X},\bar{\Theta},\bar{\Omega}) \right \| \le r$, we have
\begin{equation}\label{dg:19}
	\left \| (X,\Theta, \Omega)-(\bar{X},\bar{\Theta},\bar{\Omega})\right \|\le \kappa \dist(0, \T(X,\Theta,\Omega )).
\end{equation}

Given positive summable sequences $\left \{ \varepsilon _t \right \} $ and $\left \{ \delta _t \right \}  $ with each $\delta_t<1$, we consider the following stopping criteria  adopted in 
 \cite{li2020asymptotically}, which are implementable conditions based on the general principles in 
 \cite{rockafellar1976monotone}:
 	\begin{align*}
 		\left \|  \nabla \Gamma _t(X^{(t+1)}) \right \|  &\le \frac{\min\{\sqrt{\tau_t},1\}}{\sigma_t}\varepsilon _t; \tag{A} \\ 
 		 \left \|  \nabla \Gamma _t(X^{(t+1)}) \right \| 
 		& \le  \frac{\delta _t\min\{\sqrt{\tau_t},1\}}{\sigma_t} \left \| (X^{(t+1)},\Theta^{(t+1)},\Omega^{(t+1)})- (X^{(t)},\Theta^{(t)},\Omega^{(t)}) \right \|_{\Q_t}. \tag{B}
 	\end{align*}

Combining the above analysis with the convergence results developed in \cite[Theorem 2.3 and Theorem 2.5]{li2020asymptotically}, we present the global convergence and the local asymptotically superlinear convergence rate of Algorithm \ref{alg:dg1}, as stated in the next theorem. 

\begin{theorem}\label{th:1}
	Suppose that $\tau _t\downarrow \tau _{\infty}>0$ and $\sigma_t\uparrow\sigma_{\infty}\le \infty$.
	\begin{itemize}
		\item [(i)]\label{item:i} Let $\left \{ ( X^{(t)},\Theta^{(t)},\Omega^{(t)})\right \} $ be the sequence generated by Algorithm \ref{alg:dg1} under the stopping criterion (A). Then, the sequence $\left \{ ( X^{(t)},\Theta^{(t)},\Omega^{(t)})\right \} $ is bounded. Moreover, 
        $\left \{   X^{(t)} \right \} $ and $\left \{   \Theta^{(t)} \right \} $ converge to optimal solutions of the dual problem (\ref{dg:7}) and the primal problem(\ref{dg:3}), respectively.
		\item[(ii)]\label{item:ii} Let $\left \{ ( X^{(t)},\Theta^{(t)},\Omega^{(t)})\right \} $ be the sequence generated by Algorithm \ref{alg:dg1} under     stopping criteria (A) and (B). Assume that the initial $(X^{(0)},\Theta^{(0)},\Omega^{(0)})$ satisfies $\left \|( X^{(0)},\Theta^{(0)},\Omega^{(0)})\right. \allowbreak\left.  -\allowbreak(\bar{X},\bar{\Theta},\bar{\Omega}))\right \|_{\Q_0}\allowbreak\le r-\sum_{i=0}^{\infty } \varepsilon  _i$. Let $\kappa$ be the modulus given in (\ref{dg:19}). Then, for all $t\ge 0$, the following inequality holds:
		\begin{equation*}\small 
			\begin{aligned}
		 \left \| (X^{(t+1)},\Theta^{(t+1)},\Omega^{(t+1)}) -(\bar{X},\bar{\Theta},\bar{\Omega})\right \|_{\Q_t} 
		\le \ \zeta_t  \left \| (X^{(t )},\Theta^{(t )},\Omega^{(t )}) -(\bar{X},\bar{\Theta},\bar{\Omega})\right \|_{\Q_t},
		\end{aligned}
		\end{equation*}
		where 
		\begin{equation*}
			\zeta_t:=\frac{1}{1-\delta_t} \left ( \delta_t+\frac{\left ( 1+\delta _t \right )\kappa\gamma_t }{\sqrt[]{\sigma_t^2+\kappa^2\gamma_t^2} }  \right )\longrightarrow \zeta_{\infty }:=\frac{\kappa\gamma_{\infty }}{\sqrt[]{\sigma_{\infty}^2+\kappa^2\gamma_{\infty }^2} }<1\:\:\:as\:t\longrightarrow \infty 
			\end{equation*}
		with $\gamma_t:=\max(\tau_t,1)$ and $ \gamma_{\infty}:=\max(\tau_{\infty},1)$.
		
	\end{itemize}
\end{theorem}


\subsection{A Semismooth Newton Algorithm for the Subproblem (\ref{dg:5})}
Although Algorithm \ref{alg:dg1} exhibits favorable convergence properties, its practical performance is largely determined by the efficiency of solving subproblem (\ref{dg:5}) at each iteration. Therefore, we develop a semismooth Newton method to handle this subproblem (\ref{dg:5}) efficiently.

Notice that $\Gamma_t(\cdot)$, defined in (\ref{dg:5}), is a strongly convex and continuously differentiable function, then the unique optimal solution to problem (\ref{dg:5}) can be computed via the following nonlinear equations:
\begin{equation}\label{dg:8}
	\begin{aligned}
		\nabla \Gamma  _t(X)
		&=-\left ( \prox_{\sigma_t h} ((\Theta^1)^{(t)}-\sigma_t(X^1+S^1),\dots,\prox_{\sigma_t h}((\Theta^K)^{(t)}-\sigma_t(X^K+S^K))\right ) \\
		& \ \ +\prox_{\sigma_t \P}( \Omega^{(t)}+\sigma _tX)+\frac{\tau_t}{\sigma_t  }( X-X^{(t)} )\\
		&=0.
	\end{aligned}
\end{equation}
The function $\nabla \Gamma_t(\cdot)$ is not differentiable due to the nonsmooth proximal operator $\prox_{\sigma_t \mathcal{P}}(\cdot)$. Nevertheless, $\nabla \Gamma_t(\cdot)$ is $\gamma$-order semismooth on $\mathcal{Y}$ for any $\gamma>0$. Hence, we can apply a semismooth Newton method to solve (\ref{dg:8}). For this purpose, we define the following surrogate generalized Jacobian $\hat{\partial}^2 \Gamma  _t(X)$ of $\nabla \Gamma_t(\cdot)$ at $X$:
\begin{equation*}
	\begin{aligned}
		 \hat{\partial}^2 \Gamma  _t(X)[D]:= & \sigma_t\left ( (\phi_{\sigma_t}^+)' \left((\Theta^1)^{(t)}-\sigma_t(X^1+S^1)\right)[D^1],\dots,\right. \\[2mm]
		 &\left. (\phi_{\sigma_t}^+)' \left((\Theta^K)^{(t)}-\sigma_t(X^K+S^K)\right)[D^K]\right ) 
		 +\sigma_t\V[D]+\frac{\tau_t}{\sigma_t  }D,	\ \forall D\in\Y,
	\end{aligned}
\end{equation*}
where $\V\in\hat{\partial}\prox_{\sigma_t \P}\left ( \Omega^{(t)}+\sigma _tX \right )$. 
With these preparations in place, we now present the following semismooth Newton method to solve problem (\ref{dg:5}), as summarized in Algorithm \ref{alg:dg2}.

\begin{algorithm} 	
	\renewcommand{\thealgorithm}{1}
	\caption{A Semismooth Newton Method for Subproblem (\ref{dg:5})}\label{alg:dg2}
	\hspace*{0.01in} \raggedright {\bf Input:} $\bar{\mu}\in(0,1/2)$, $\bar{\tau}\in(0,1]$, $\bar{\gamma}$, $\bar{\delta}\in(0,1)$, $X^{(t,0)}\in\Y_{++}$. Set $m=0$.\\
	\begin{algorithmic}[1]
		\State {Choose  $(\V^k)^{(t)}\in\hat{\partial}\prox_{\sigma_t h} \left((\Theta^k)^{(t)}-\sigma_t((X^k)^{(t,m)}+S^k)\right)$, $k=1,\dots,K$, and $\W^{(t)}\in\allowbreak\hat{\partial}\prox_{\sigma_t \P}\left ( \Omega^{(t)}+\sigma _tX^{(t,m)} \right )$. Let $U^{(t)}:=\sigma_t((\V^1)^{(t)},\dots,(\V^K)^{(t)})+\sigma_t\W^{(t)}+\frac{\tau_t}{\sigma_t  }\I$. Solving the following linear system
		\begin{equation*}
			U^{(t)}[D]=-\nabla \Gamma  _t(X^{(t,m)})
		\end{equation*}
		by the conjugate gradate (CG) algorithm to find $D^{(m)}$ such that $$\left \|U^{(t)}[D^{(m)}]+\nabla \Gamma  _t(X^{(t,m)})  \right \| \le\min(\bar{\gamma}, \left \| \nabla \Gamma  _t(X^{(t,m)}) \right \|^{1+\bar{\tau}}  ). $$}
		
		\State{Set $\alpha_m=\bar{\delta}^{n_m}$, where $n_m$ is the first nonnegative integer $n$ for which
		\begin{equation*}
			\Gamma  _t(X^{(t,m)}+\bar{\delta}^{n}D^{(m)})\le	\Gamma  _t(X^{(t,m)})+\bar{\mu}\bar{\delta}^{n}\left \langle  \nabla \Gamma  _t(X^{(t,m)}), D^{(m)}\right \rangle .
		\end{equation*} }
		
		\State Update  $X^{(t,m+1)}=X^{(t,m)}+\alpha_mD^{(m)}$, $m\leftarrow m+1$, and go to step 1.
	\end{algorithmic}
\end{algorithm}		

\begin{remark}
In practical implementations, we set $\bar{\gamma} = 1$ when $t < 2$ (the main proximal ALM iteration) and $m < 5$ (the semismooth Newton iteration for solving the subproblem); otherwise, we set $\bar{\gamma} = 10^{-1}$. 
The parameter $\bar{\tau}$ is fixed at $10^{-1}$. 
In the line search procedure, we choose $\bar{\mu} = 10^{-4}$ and $\bar{\delta} = 0.5$.
\end{remark}

The global convergence and the local convergence rate of the semismooth Newton method are extensively studied in the existing literature \cite[Proposition 3.3 and Theorem 3.4]{zhao2010newton} and  
 \cite[Theorem 3]{li2018efficiently}. For completeness, the main result is repeated in the following.
\begin{theorem}\label{Th:3}
	Let $\{X^{(t)}\}$ be the sequence generated by Algorithm \ref{alg:dg2}. Then, $\{X^{(t)}\}$ converges to the unique optimal solution $\hat{X}$ of subproblem (\ref{dg:8}). In addition, for all $t$ sufficiently large, 
	\begin{equation}\label{dg:28}
		\left \| X^{(t+1)}-\hat{X}  \right \| =\mathcal{O} \left(\left \| X^{(t)}-\hat{X} \right \|^{1+\bar{\tau } } \right),
	\end{equation}
	where $\bar{\tau}\in(0,1]$ is the constant used in Algorithm \ref{alg:dg2}.
\end{theorem}

In essence, for sufficiently large $t$, (\ref{dg:28}) implies the existence of a neighborhood around the solution $\hat{X}$ such that, once the iterates enter this region, they converge to the solution at a superlinear rate of order $1+\bar{\tau}$. In practice, only a few additional iterations are typically required to achieve high-precision solutions. Moreover, the global convergence established in Theorem \ref{Th:3} ensures that the generated sequence will eventually reach this neighborhood, thus exhibiting rapid local convergence.

\section{Experimental results}\label{exp}
In this section, we present numerical results that compare the performance of our proposed algorithm, {\sc Ssnpal},  with two existing methods: ADMM and the proximal Newton-type method proposed in \cite{yang2015fused}, referred to as MGL. All algorithms are implemented in MATLAB (version 9.10, R2021a) and executed on a Linux server (32-core, Intel Xeon Gold 6326 @ 2.90GHz, 251 GB of RAM).

\subsection{Implementation Details of Benchmark Methods}
In this subsection, we  provide a brief overview of the benchmark algorithms, including ADMM and MGL, used for comparison with our proposed method.
\subsubsection{ADMM}
ADMM is used to solve the dual problem (\ref{dg:7}), which can be written as:
\begin{equation}\label{dg:21}
\begin{aligned}
	\underset{X,Y}{\min}  &\:\: -\sum_{k=1}^{K}\log\det(Y^k) +\P^{*}(X)\\
\text{s.t.} &\:\:X+S=Y.
\end{aligned}
\end{equation}

For $\sigma>0$, the augmented Lagrangian function associated with (\ref{dg:21}) is defined by
\begin{equation*}
\hat{\L}_{\sigma } (X,Y;\Theta)=-\sum_{k=1}^{K}\log\det(Y^k)+\P^*(X)-\left \langle \Theta,X+S-Y \right \rangle +\frac{\sigma }{2}\left \| X+S-Y \right \| ^2.
\end{equation*}
The ADMM iterative scheme for (\ref{dg:21}) is as follows:  given $\tau\in(0,(1+\sqrt{5})/2)$ and an initial point $(X^{(0)},Y^{(0)},\Theta^{(0)})\in\Y_{++}\times\Y_{++}\times\Y_{++}$, the $t$-th iteration is given by
\begin{equation*}
\left\{\begin{matrix}
	X^{(t+1)}=\underset{X}{\arg\min}\hat{\L}_{\sigma }(X,Y^{(t)};\Theta^{(t)}), \:\:\:\:\:\:\:\:\:\:\:\:\:\:\:\:\\[2mm]
	Y^{(t+1)}=\underset{Y}{\arg\min}\hat{\L}_{\sigma }(X^{(t+1)},Y;\Theta^{(t)}),\:\:\:\:\:\:\:\:\:\:\: \\[2mm]
	\Theta^{(t+1)}=\Theta^{(t)}-\tau \sigma (X^{(t+1)}+S-Y^{(t+1)}).
\end{matrix}\right.
\end{equation*}
Here, $X^{(t+1)}$ is updated by $X^{(t+1)}=Y^{(t)}+\Theta^{(t)}/\sigma-S-\prox_{\P}(Y^{(t)}+\Theta^{(t)}/\sigma-S)$ and $Y^{(t+1)}=((Y^1)^{(t+1)},\dots,(Y^K)^{(t+1)})$ is updated by $(Y^k)^{(t+1)}=\phi_{\frac{1}{\sigma}}^{+}((X^k)^{(t+1)}+S^k-(\Theta^k)^{(t)})/\sigma$, $k=1,\dots,K.$


 
Different from the ADMM implemented by Yang et al. \cite{yang2015fused}, which uses a fixed penalty parameter $\sigma$ and step length $\tau=1$,  we set the step length $\tau = 1.618$ and adjust the parameter $\sigma$ based on the ratio $\gamma^t$  of primal 
to dual  feasibilities at iteration $t$ \cite[Section 4.4]{lam2018fast}. Specifically, we set $\sigma^{t+1}=1.2\sigma^t$  if $\gamma^t>5$,  $\sigma^{t+1}=\frac{\sigma^t}{1.2}$  if $\frac{1}{\gamma^t}>5$, and leave $\sigma$ unchanged otherwise.
The larger step length and the adaptive update of $\sigma$ can empirically accelerate the convergence speed.

\subsubsection{Implementation of MGL}
In \cite{yang2015fused}, Yang et al. proposed an active set identification strategy for MGL  to determine the subset of variables to be updated during each Newton iteration. However, it is important to note that the strategy described in 
 \cite[Lemma 4.2]{yang2015fused} was originally designed for the fused regularization form. In this paper, we establish a similar identification result for the clustered regularization, and the proof is omitted for brevity.

\begin{theorem}\label{lem2}
In the $t$-th iteration, define the active set $\B$ by
	\begin{equation*}
		\begin{aligned}
			\B=&\left \{ (i,j)|\left ( \Theta_{[ij]} \right )^{(t)} =\mathbf{0}, \:\: \left | \sum_{k\in \A }^{} \left [S^k_{ij}-((\Theta^k)^{(t)})^{-1}_{ij}\right ]  \right |\le\left | \A \right |\mu_1+\left | \A \right |(K-\left | \A \right |) \mu_2, \right. \\[2mm]
			&\:\:\:\left. \forall\; \emptyset \ne \A \subseteq \left \{ 1,\dots,K \right \}  \right \}.	
			\end{aligned}
	\end{equation*}
	Then, $\Delta^* =0$ is the optimal solution to the following problem:
	\begin{equation*}
		\begin{aligned}
			\underset{\Delta }{\min}  &\:\: \tilde{L}_t(\Theta^{(t)}+\Delta ) +\P(\Theta^{(t)}+\Delta ),\\
			\st &\:\:(\Delta ^1)_{ij}=\cdots=(\Delta ^K)_{ij}=0,\ (i,j)\in \F,
		\end{aligned}
	\end{equation*}
	where  $\tilde{L}_t(\cdot)$ is the quadratic approximation of $L(\cdot)$ at the current iterate point and $\F$ is the complement of $\B$.
\end{theorem}

Theorem \ref{lem2} implies that when the variables in
the free set $\F$ are fixed, no update is needed for the variables in the active set $\B$. The remaining iterative details of 
MGL with active set identification used for solving the cMGGMs can be found in 
 \cite[Algorithm1]{yang2015fused}.

\subsection{Experimental settings} 
We  present several key implementation details of all tested algorithms,  which are crucial for ensuring numerical stability and achieving efficient convergence in practical applications.

\subsubsection{Stopping criteria}
For   problem (\ref{dg:3}), the accuracy
of the approximate optimal solution is measured  as below. 
Let $\tol =10^{-6}$ be a given tolerance  in the following experiments. 
\begin{itemize}
	\item The {\sc Ssnpal} is terminated if $\eta_S<\tol$, where 
	\begin{equation*}
		\begin{aligned}
			\eta_S:=\max &\left\{  \frac{\left \|   \Omega-\prox_{\P}(X+\Omega)\right \|}{1+\left \| \Omega \right \| },\;\;\max_{1\le k\le K} \left \{  \frac{\left \|\Theta^k(X^k+S^k)-\I\right \|}{1+\sqrt{p} } \right \} , \right.\\[2mm]
			&\:\:\:\left. \frac{\left \| \Theta-\Omega \right \| }{1+\left \| \Theta \right \| },\;\;\frac{\text{pobj}_{\text{S}}-\text{dobj}_{\text{S}}}{1+|\text{pobj}_{\text{S}}|+|\text{dobj}_{\text{S}}|} \right\}.
			\end{aligned}				
	\end{equation*}
	Here, $\text{pobj}_{\text{S}}$ and $\text{dobj}_{\text{S}}$ denote the primal and dual objective values obtained by {\sc Ssnpal}.
	\item The ADMM is terminated when $\eta_A<\tol$, where
	\begin{equation*}
		\begin{aligned}
	\eta_A:=\max&\left\{ \frac{\left \|   \Theta-\prox_{\P}(X+\Theta)\right \|}{1+\left \| \Theta \right \| },\;\;\max_{1\le k\le K} \left \{  \frac{\left \|\Theta^k Y^k-\I\right \|}{1+\sqrt{p} } \right \},\right.\\[2mm]
	&\:\:\:\left.\frac{\left \| X+S-Y \right \| }{1+\left \| S \right \| },\;\;\frac{\text{pobj}_{\text{A}}-\text{dobj}_{\text{A}}}{1+|\text{pobj}_{\text{A}}|+|\text{dobj}_{\text{A}}|} \right\}.				
		\end{aligned}				
\end{equation*}
	Here, $\text{pobj}_{\text{A}}$ and $\text{dobj}_{\text{A}}$ are the primal and dual objective values obtained by ADMM.
	\item The MGL is terminated when the relative difference between  its objective value and the primal objective value obtained by the {\sc Ssnpal} is less than the
	given tolerance $\tol$, i.e.,
	\begin{equation*}
		\eta_M:=\frac{ \text{pobj}_{\text{M}}-\text{pobj}_{\text{S}} }{1+|\text{pobj}_{\text{M}}|+|\text{pobj}_{\text{S}}|}<\tol,
	\end{equation*}
where $\text{pobj}_{\text{M}}$ denotes the objective value returned by MGL.
\end{itemize}

To ensure fairness in comparison, the maximum runtime for each algorithm is capped at one hour. 
When an algorithm does not reach the prescribed accuracy within this time limit, 
its corresponding entry in the tables is marked by ``-''.

\subsubsection{Warm-start strategy of {\sc Ssnpal}}
 To enhance the overall efficiency of the proposed algorithm, we employ a warm-start strategy during its initialization phase. Specifically, rather than initializing from arbitrary points, we first run a few iterations of the ADMM—starting from $K$ identity matrices—to generate an approximate solution with moderate accuracy. This solution is then used as the starting point for the {\sc Ssnpal}.
 Although  ADMM is well-suited for obtaining the solutions with low to moderate accuracy efficiently, it may slow down significantly when higher precision is required.       
 In contrast,  {\sc Ssnpal} exhibits favorable local convergence properties— asymptotically superlinear  convergence near the optimum—but tends to be more computationally demanding  when the iterates are still far from optimality.          Therefore, we employ  ADMM to first reduce the optimality residual below $100\tol$, or to stop early when the residual falls below $400\tol$ and the number of iterations exceeds 800.
 This strategy improves numerical stability and significantly reduces   total runtime.

\subsubsection{Model selection}
We adopt $3$-fold cross validation to select the tuning parameters $\mu_1$ and $\mu_2$ in (\ref{dg:3}), which control the sparsity and similarity among estimated graphs, respectively. Specifically, we define the turning grids $\mu_{1,\text{grid}}$ and $\mu_{2,\text{grid}}$ for $\mu_1$ and $\mu_2$, and evaluate each parameter pair using the following procedure.
 
For each fold $n = 1, \dots, 3$, let $\hat{S}_n$ denote the sample covariance matrix computed from the $n$-th segment, and let $\hat{\Theta}^{(\mu_1,\mu_2)}_{-n}$ be the inverse covariance matrix estimated from the remaining data  using  parameters $\mu_1$ and $\mu_2$. The performance of each $(\mu_1, \mu_2)$ is evaluated by the average predictive negative log-determinant:
\begin{equation*}
\text{CV}(\mu_1,\mu_2)=\sum_{n=1}^{3}\left ( -\log\det\hat{\Theta}^{(\mu_1,\mu_2)}_{-n} +\left \langle \hat{S}_{n}, \hat{\Theta}^{(\mu_1,\mu_2)}_{-n}\right \rangle  \right ).
\end{equation*}
The optimal tuning parameters $(\hat{\mu}_1, \hat{\mu}_2)$ are determined by minimizing the cross validation score defined above.

\subsection{Numerical results}
In this subsection, we compare the numerical performance of {\sc Ssnpal}, ADMM, and MGL on both synthetic and real  data sets.

\subsubsection{Simulation}

%
%
 \begin{figure}[htbp]
	\centering
	
	\begin{minipage}[t]{0.28\textwidth}
		\centering
		\includegraphics[width=\linewidth]{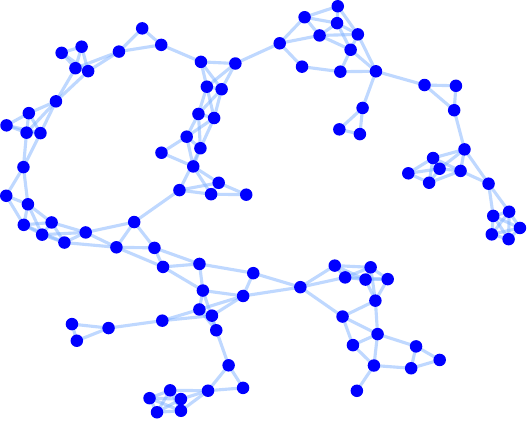} \\
		(a) Nearest-neighbor network
	\end{minipage}
	\hspace{2em}
	\begin{minipage}[t]{0.28\textwidth}
		\centering
		\includegraphics[width=\linewidth]{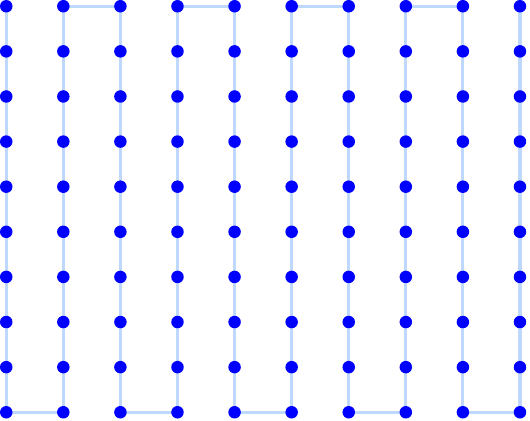}\\
		(b) Chain network
	\end{minipage}
	\hspace{2em}
	\begin{minipage}[t]{0.28\textwidth}
		\centering
		\includegraphics[width=\linewidth]{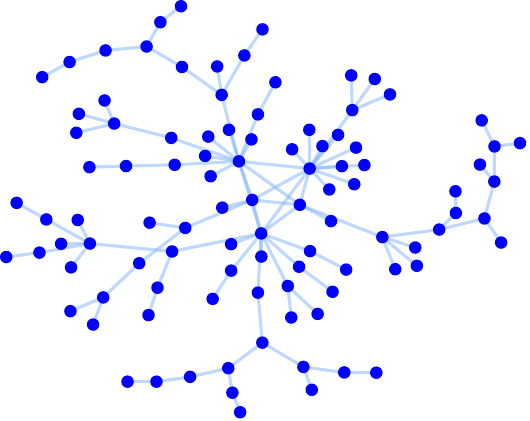}\\
		(c) Scale free network
	\end{minipage}
	
	\caption{The common links present in all categories in the three simulated networks.}
	\label{fig:three_networks}
\end{figure}
As illustrated in Figure \ref{fig:three_networks}, we generate three types of simulated networks to evaluate the performance of all tested algorithms: Case I: nearest-neighbor network; Case II: chain network; Case III: scale free network.   
 The three types of precision matrices  $\Theta^k= \left(\Sigma^{k}\right)^{-1}$, $k=1,\dots,K$, are generated  following the procedure described in \cite{guo2011joint}.
 For the case where the true precision matrices $\Theta^k$, $k=1,\dots,K$,  have a block diagonal structure with $T$ blocks, each block $(\Theta_t^1, \dots, \Theta_t^K)$, $t = 1, \dots, T$,  is generated independently using the same procedure.
 For each $k=1,\dots,K$, we generate $N=10p$ independently and identically distributed observations from a multivariate Gaussian distribution $\N(0, \Sigma^k)$.

\begin{table}[htbp]
	\footnotesize
	\caption{\footnotesize {The performance of {\sc Ssnpal}, ADMM, and MGL on three simulated networks. In the table, ``a'' = {\sc Ssnpal}, ``b'' = ADMM, and  ``c'' = MGL. Times are shown in seconds.}}\label{tab:1-1}
	\begin{tabular}{l@{\hspace{5pt}} l@{\hspace{5pt}} l@{\hspace{5pt}}l @{\hspace{5pt}}l @{\hspace{10pt}}c@{\hspace{1pt}}c@{\hspace{1pt}}c@{\hspace{1pt}}c@{\hspace{1pt}}c@{\hspace{30pt}}c@{\hspace{1pt}}c@{\hspace{1pt}}c@{\hspace{1pt}}c@{\hspace{1pt}}c } \toprule
		\multirow{2}{*}{Problem} & 
		
		\multirow{2}{*}{Case} & 
		
		\multirow{2}{*}{$\mu_1$} & 
		\multirow{2}{*}{$\mu_2$} & 
		\multirow{2}{*}{Density} & 
		
		\multicolumn{5}{c}{$R_{\text{KKT}}$} & 
		\multicolumn{5}{c}{Time} \\
		\cmidrule(r ){6-10} \cmidrule( ){11-15}  
		(p;K)& &&&&a &$|$& b &$|$&c & a&$|$&b  &$|$& c        \\

		\hline
		
		(500;5)& I & $0.025$ & $0.01$ & 0.012 & 3.34e-07 &$|$& 9.90e-07&$|$& 4.74e-07 & 6.82 &$|$& 11.96&$|$& 9.85 \\
 & II & $0.025$ & $0.01$ & 0.007 & 2.99e-07 &$|$& 9.02e-07&$|$& 2.52e-07 & 4.56 &$|$& 7.48&$|$& 9.47 \\
 & III & $0.025$ & $0.01$ & 0.007 & 4.49e-07 &$|$& 9.80e-07&$|$& 1.94e-07 & 5.01 &$|$& 8.84&$|$& 9.74 \\
\hline
(500;10)& I & $0.015$ & $0.01$ & 0.013 & 2.88e-07 &$|$& 8.74e-07&$|$& 2.61e-07 & 11.87 &$|$& 16.44&$|$& 47.33 \\
 & II & $0.015$ & $0.01$ & 0.007 & 3.31e-07 &$|$& 8.59e-07&$|$& 2.79e-07 & 9.77 &$|$& 15.40&$|$& 47.90 \\
 & III & $0.015$ & $0.01$ & 0.007 & 2.37e-07 &$|$& 8.56e-07&$|$& 9.75e-07 & 10.17 &$|$& 14.63&$|$& 41.82 \\
\hline
(500;20)& I & $0.015$ & $0.005$ & 0.011 & 2.75e-07 &$|$& 9.17e-07&$|$& 2.63e-07 & 29.33 &$|$& 45.62&$|$& 99.23 \\
 & II & $0.01$ & $0.005$ & 0.006 & 3.26e-07 &$|$& 8.47e-07&$|$& 3.86e-07 & 27.99 &$|$& 39.81&$|$& 88.08 \\
 & III & $0.01$ & $0.005$ & 0.007 & 2.56e-07 &$|$& 9.69e-07&$|$& 2.41e-07 & 26.66 &$|$& 38.24&$|$& 89.00 \\
\hline
(1000;5)& I & $0.015$ & $0.01$ & 0.006 & 4.31e-07 &$|$& 9.75e-07&$|$& 4.53e-07 & 25.09 &$|$& 40.33&$|$& 40.67 \\
 & II & $0.015$ & $0.01$ & 0.004 & 3.84e-07 &$|$& 8.64e-07&$|$& 3.21e-07 & 22.46 &$|$& 29.62&$|$& 40.86 \\
 & III & $0.015$ & $0.01$ & 0.004 & 3.20e-07 &$|$& 9.78e-07&$|$& 2.96e-07 & 22.12 &$|$& 31.01&$|$& 40.37 \\
\hline
(1000;10)& I & $0.015$ & $0.005$ & 0.006 & 5.17e-07 &$|$& 9.29e-07&$|$& 4.88e-07 & 63.04 &$|$& 97.12&$|$& 209.38 \\
 & II & $0.01$ & $0.005$ & 0.004 & 3.73e-07 &$|$& 8.53e-07&$|$& 2.90e-07 & 46.01 &$|$& 72.66&$|$& 218.71 \\
 & III & $0.01$ & $0.005$ & 0.004 & 3.51e-07 &$|$& 9.23e-07&$|$& 3.66e-07 & 55.16 &$|$& 77.13&$|$& 214.20 \\
\hline
(1000;20)& I & $0.01$ & $0.005$ & 0.006 & 5.84e-07 &$|$& 8.76e-07&$|$& 2.88e-07 & 195.31 &$|$& 219.86&$|$& 448.44 \\
 & II & $0.01$ & $0.005$ & 0.003 & 4.43e-07 &$|$& 9.95e-07&$|$& 2.89e-07 & 117.74 &$|$& 184.11&$|$& 470.94 \\
 & III & $0.01$ & $0.005$ & 0.003 & 2.35e-07 &$|$& 8.56e-07&$|$& 9.86e-07 & 139.22 &$|$& 183.57&$|$& 410.10 \\
\bottomrule

	\end{tabular}
	\end{table}

 \begin{sidewaystable}[htbp]
	\footnotesize
	\centering
	\renewcommand\arraystretch{1.1}
	\caption{\footnotesize {	
			The performance of {\sc Ssnpal}, ADMM, and MGL with and without screening on three simulated networks. 
			{\sc Ssnpal}-S, ADMM-S, and MGL-S are {\sc Ssnpal}, ADMM, and MGL with
			screening, respectively.
			In the table, ``a'' = {\sc Ssnpal}, ``b'' = ADMM,  ``c'' = MGL, ``d'' = {\sc Ssnpal}-S, ``e'' = ADMM-S,  and  ``f'' = MGL-S. Times are shown in seconds.
		}}
	\label{tab:2}
	\begin{tabular}{l@{\hspace{7pt}} l@{\hspace{7pt}} l@{\hspace{5pt}}l@{\hspace{7pt}} l @{\hspace{7pt}}c@{\hspace{2pt}}c@{\hspace{2pt}}c@{\hspace{2pt}}c@{\hspace{2pt}}c@{\hspace{2pt}}c@{\hspace{2pt}}c@{\hspace{2pt}}c@{\hspace{2pt}}c@{\hspace{2pt}}c@{\hspace{2pt}}c@{\hspace{10pt}}c@{\hspace{1pt}}c@{\hspace{1pt}}c@{\hspace{1pt}}c@{\hspace{1pt}}c@{\hspace{1pt}}c@{\hspace{1pt}}c@{\hspace{1pt}}c@{\hspace{1pt}}c@{\hspace{1pt}}c@{\hspace{1pt}}c}
		\hline
		\multirow{2}{*}{$(p;K;T)$} & 
		
		\multirow{2}{*}{Case} & 
		
		\multirow{2}{*}{$\mu_1$} & 
		\multirow{2}{*}{$\mu_2$} & 
		\multirow{2}{*}{Density} & 
		
		\multicolumn{11}{c}{$R_{\text{KKT}}$} & 
		\multicolumn{11}{c}{Time} \\
		\cmidrule(r ){6-16} \cmidrule( ){17-27} 
		& &&&&a &$|$ &b& $|$& c& $|$ &d &$|$ &e& $|$& f &a& $|$& b &$|$ &c &$|$& d &$|$ &e& $|$& f        \\
		
		\hline
		(500;10;5)& I & $0.03$ & $0.01$ & 0.011 & 1.42e-07&$|$& 9.47e-07&$|$& 9.10e-07&$|$& 4.92e-07&$|$& 1.00e-06&$|$& 8.92e-07 & 13.36&$|$& 21.16&$|$& 39.48&$|$& 3.36&$|$& 6.52&$|$& 8.32 \\
 & II & $0.025$ & $0.01$ & 0.006 & 2.88e-07&$|$& 9.22e-07&$|$& 2.85e-07&$|$& 3.06e-07&$|$& 9.33e-07&$|$& 3.05e-07 & 8.52&$|$& 15.18&$|$& 45.84&$|$& 2.60&$|$& 4.71&$|$& 9.84 \\
 & III & $0.025$ & $0.01$ & 0.007 & 2.01e-07&$|$& 8.90e-07&$|$& 9.17e-07&$|$& 3.18e-07&$|$& 9.69e-07&$|$& 9.35e-07 & 10.22&$|$& 15.88&$|$& 40.43&$|$& 2.83&$|$& 5.01&$|$& 8.56 \\
\hline
(500;10;10)& I & $0.025$ & $0.01$ & 0.011 & 2.19e-07&$|$& 9.82e-07&$|$& 2.38e-07&$|$& 6.97e-07&$|$& 9.80e-07&$|$& 9.87e-07 & 12.36&$|$& 18.46&$|$& 45.90&$|$& 2.01&$|$& 3.58&$|$& 4.50 \\
 & II & $0.02$ & $0.01$ & 0.006 & 5.45e-07&$|$& 9.15e-07&$|$& 3.04e-07&$|$& 7.58e-07&$|$& 9.87e-07&$|$& 3.25e-07 & 9.23&$|$& 14.30&$|$& 46.03&$|$& 2.28&$|$& 3.43&$|$& 6.85 \\
 & III & $0.02$ & $0.01$ & 0.006 & 2.34e-07&$|$& 9.10e-07&$|$& 2.58e-07&$|$& 6.39e-07&$|$& 9.12e-07&$|$& 9.95e-07 & 10.75&$|$& 13.18&$|$& 46.43&$|$& 2.77&$|$& 3.75&$|$& 8.74 \\
\hline
(500;20;5)& I & $0.02$ & $0.01$ & 0.010 & 1.85e-07&$|$& 9.34e-07&$|$& 3.65e-07&$|$& 7.38e-07&$|$& 9.81e-07&$|$& 4.10e-07 & 26.04&$|$& 32.54&$|$& 82.02&$|$& 6.38&$|$& 8.56&$|$& 14.86 \\
 & II & $0.02$ & $0.01$ & 0.006 & 3.82e-07&$|$& 9.12e-07&$|$& 3.07e-07&$|$& 6.82e-07&$|$& 9.13e-07&$|$& 3.22e-07 & 23.46&$|$& 38.98&$|$& 97.39&$|$& 6.05&$|$& 9.88&$|$& 17.18 \\
 & III & $0.02$ & $0.01$ & 0.006 & 2.29e-07&$|$& 9.41e-07&$|$& 4.06e-07&$|$& 3.42e-07&$|$& 9.73e-07&$|$& 4.02e-07 & 26.34&$|$& 32.23&$|$& 85.61&$|$& 5.79&$|$& 8.30&$|$& 14.77 \\
\hline
(500;20;10)& I & $0.02$ & $0.005$ & 0.011 & 2.64e-07&$|$& 9.05e-07&$|$& 2.39e-07&$|$& 6.05e-07&$|$& 1.00e-06&$|$& 9.95e-07 & 36.40&$|$& 41.74&$|$& 92.79&$|$& 3.97&$|$& 6.52&$|$& 8.46 \\
 & II & $0.02$ & $0.005$ & 0.006 & 3.82e-07&$|$& 9.12e-07&$|$& 2.94e-07&$|$& 8.07e-07&$|$& 9.89e-07&$|$& 3.28e-07 & 23.64&$|$& 38.89&$|$& 97.43&$|$& 3.71&$|$& 6.06&$|$& 8.95 \\
 & III & $0.02$ & $0.005$ & 0.006 & 2.75e-07&$|$& 8.81e-07&$|$& 9.37e-07&$|$& 4.29e-07&$|$& 8.84e-07&$|$& 9.84e-07 & 21.70&$|$& 36.12&$|$& 85.50&$|$& 3.43&$|$& 5.75&$|$& 7.88 \\
\hline
(1000;10;5)& I & $0.025$ & $0.01$ & 0.006 & 2.06e-07&$|$& 9.13e-07&$|$& 2.31e-07&$|$& 7.66e-07&$|$& 9.90e-07&$|$& 9.97e-07 & 58.65&$|$& 90.60&$|$& 208.57&$|$& 10.59&$|$& 19.12&$|$& 32.68 \\
 & II & $0.02$ & $0.01$ & 0.003 & 3.99e-07&$|$& 9.69e-07&$|$& 3.13e-07&$|$& 7.21e-07&$|$& 9.89e-07&$|$& 3.45e-07 & 50.20&$|$& 67.81&$|$& 213.27&$|$& 10.07&$|$& 14.62&$|$& 37.04 \\
 & III & $0.02$ & $0.01$ & 0.003 & 2.14e-07&$|$& 8.80e-07&$|$& 9.74e-07&$|$& 5.10e-07&$|$& 8.83e-07&$|$& 9.79e-07 & 51.56&$|$& 61.85&$|$& 188.29&$|$& 10.79&$|$& 13.42&$|$& 32.33 \\
\hline
(1000;10;10)& I & $0.02$ & $0.01$ & 0.006 & 2.60e-07&$|$& 9.26e-07&$|$& 2.95e-07&$|$& 8.49e-07&$|$& 9.82e-07&$|$& 9.95e-07 & 55.66&$|$& 75.18&$|$& 211.83&$|$& 6.12&$|$& 9.72&$|$& 18.55 \\
 & II & $0.015$ & $0.01$ & 0.003 & 5.69e-07&$|$& 1.00e-06&$|$& 2.93e-07&$|$& 8.69e-07&$|$& 9.98e-07&$|$& 3.83e-07 & 44.31&$|$& 66.28&$|$& 214.56&$|$& 9.17&$|$& 11.81&$|$& 30.65 \\
 & III & $0.015$ & $0.01$ & 0.003 & 2.72e-07&$|$& 9.70e-07&$|$& 2.04e-07&$|$& 5.41e-07&$|$& 9.63e-07&$|$& 2.33e-07 & 46.55&$|$& 64.74&$|$& 214.81&$|$& 5.31&$|$& 8.45&$|$& 19.76 \\
\hline
(1000;20;5)& I & $0.015$ & $0.01$ & 0.005 & 2.00e-07&$|$& 8.04e-07&$|$& 3.36e-07&$|$& 2.65e-07&$|$& 9.46e-07&$|$& 3.68e-07 & 142.57&$|$& 171.59&$|$& 408.51&$|$& 21.79&$|$& 28.41&$|$& 54.89 \\
 & II & $0.015$ & $0.005$ & 0.003 & 4.13e-07&$|$& 9.93e-07&$|$& 3.13e-07&$|$& 9.17e-07&$|$& 9.98e-07&$|$& 3.43e-07 & 146.20&$|$& 188.58&$|$& 473.18&$|$& 20.54&$|$& 33.04&$|$& 62.36 \\
 & III & $0.015$ & $0.005$ & 0.003 & 2.29e-07&$|$& 9.60e-07&$|$& 2.34e-07&$|$& 9.60e-07&$|$& 9.60e-07&$|$& 9.53e-07 & 142.11&$|$& 178.99&$|$& 473.78&$|$& 19.54&$|$& 31.25&$|$& 55.81 \\
\hline
(1000;20;10)& I & $0.015$ & $0.005$ & 0.006 & 2.72e-07&$|$& 9.86e-07&$|$& 2.65e-07&$|$& 5.02e-07&$|$& 1.00e-06&$|$& 3.03e-07 & 152.46&$|$& 212.98&$|$& 451.79&$|$& 14.26&$|$& 22.33&$|$& 34.17 \\
 & II & $0.01$ & $0.005$ & 0.003 & 5.21e-07&$|$& 9.94e-07&$|$& 3.66e-07&$|$& 8.07e-07&$|$& 9.89e-07&$|$& 3.92e-07 & 133.85&$|$& 188.25&$|$& 475.51&$|$& 15.12&$|$& 22.00&$|$& 40.26 \\
 & III & $0.015$ & $0.005$ & 0.003 & 2.41e-07&$|$& 9.49e-07&$|$& 9.84e-07&$|$& 4.80e-07&$|$& 9.54e-07&$|$& 9.98e-07 & 144.17&$|$& 180.78&$|$& 417.13&$|$& 12.51&$|$& 19.25&$|$& 33.34 \\
\bottomrule

	\end{tabular}
\end{sidewaystable}

 Table \ref{tab:1-1} reports the performance of the four algorithms, namely {\sc Ssnpal}, ADMM, and MGL, in terms of the relative KKT residuals ($R_{\text{KKT}}$) and the computational time (Time).
 	The regularization parameters $\mu_1$  and $\mu_2$ are selected via 3-fold cross validation over  the grids $\mu_{1,\text{grid}}=\{0.03,0.025,\dots,0.005\}$ and $\mu_{2,\text{grid}}=\{0.02,0.015, 0.01,0.005\}$, respectively. 
 	It can be observed that MGL generally requires substantially more computing time than the other algorithms when $R_{\text{KKT}}\le 10^{-6}$, particularly for $K=10$ and $K=20$, where the performance gap between MGL and {\sc Ssnpal} becomes more pronounced. 
 	{\sc Ssnpal} generally achieves the best computational performance  across all tested problem dimensions,  suggesting good scalability.  	
 To access the computational efficiency of the proposed screening rule,
we incorporate it into three existing algorithms and conduct experiments on simulated networks with a block diagonal structure of $T$ blocks. The regularization parameters $\mu_1$  and $\mu_2$ are selected via 3-fold cross validation over  the grids $\mu_{1,\text{grid}}=\{0.05,0.045,\dots,0.005\}$ and $\mu_{2,\text{grid}}=\{0.02,0.015,0.01,0.005\}$, respectively. 
The corresponding numerical results are presented in Table \ref{tab:2}.  
As shown in the table, the introduction of the screening strategy significantly accelerates all three algorithms, demonstrating its effectiveness in reducing computational complexity.
Nevertheless, the degree of acceleration differs across experiments because the screening rule (\ref{dg:20}) is affected by the parameters $\mu_1$ and $\mu_2$.
When these parameters take relatively large values, this rule can identify a greater number of block diagonal components, resulting in significant speedups.
Conversely, smaller values of $\mu_1$ and $\mu_2$ lead to fewer detected blocks—or even a single block—thereby diminishing the acceleration effect.

\subsubsection{S \& P 500 Stock Price}
In this subsection, the performance of {\sc Ssnpal}, ADMM, and MGL  is evaluated on several data sets constructed from the Standard \& Poor’s 500 stock price\footnote{Available at \url{http://www.yahoo.com}.}, spanning diverse time periods and structural patterns.

We begin with a relatively short time span from January 2004 to December 2006, which yields 755 daily returns for a total of 370 stocks. This data set,  denoted as SPX3a,  is divided into $K = 3$ annual intervals. To address the limited sample size in each period and  improve interpretability, we consider two randomly selected subsets of stocks, with sizes $p = 100$ and $p = 200$.
Moreover, we further evaluate the performance over an extended time horizon using a ten year data set, named SPX10b, covering the years 2007 to 2016. This period includes 2769 daily returns for 272 consistently active stocks (those with uninterrupted price records throughout the time frame). It is divided into $K = 10$ one-year segments, and we again analyze random subsets of size $p = 100$ and $p = 200$.
We also introduce SPX13c, covering the entire range from January 2004 to December 2016, and similarly analyze two randomly selected subsets of stocks with sizes $p = 100$ and $p = 200$.

In addition, a sector-based data set, referred to as SPX13d, is constructed to  study the evolution of direct interactions among stocks, particularly in response to the global financial crisis. The SPX13d includes 212 stocks drawn from five major sectors, as defined by the Global Industry Classification Standard (GICS): 41 from Information Technology, 41 from Consumer Discretionary, 43 from Health Care, 39 from Financials, and 48 from Industrials. It spans the same period as SPX13c (from January 1, 2004, to December 31, 2016) and includes the 2007–2009 global financial crisis.

\begin{table}[htbp]
	\footnotesize
	\caption{\footnotesize {Performance of {\sc Ssnpal}, ADMM, and MGL on stock price data. In the table, ``a'' = {\sc Ssnpal}, ``b'' = ADMM,  and  ``c'' = MGL. Times are shown in seconds.}}\label{tab:22} 
		\begin{tabular}{l@{\hspace{10pt}} l@{\hspace{5pt}} l@{\hspace{10pt}}l @{\hspace{10pt}}c@{\hspace{2pt}}c@{\hspace{2pt}}c@{\hspace{2pt}}c@{\hspace{2pt}}c@{\hspace{23pt}}c@{\hspace{2pt}}c@{\hspace{2pt}}c@{\hspace{2pt}}c@{\hspace{2pt}}c } \toprule
			\multirow{2}{*}{Problem} & 
			\multirow{2}{*}{$\mu_1$} & 
			\multirow{2}{*}{$\mu_2$} & 
			\multirow{2}{*}{Density} & 
			
			\multicolumn{5}{c}{$R_{\text{KKT}}$} & 
			\multicolumn{5}{c}{Time} \\
			\cmidrule(r ){5-9} \cmidrule( ){10-14}  
			(p;K)& &&&a &$|$& b &$|$&c & a&$|$&b &$|$& c         \\		
			\hline
		 
			SPX3a100 & $5 \cdot 10^{-4}$ & $5 \cdot 10^{-5}$ & 0.019 & 6.41e-07&$|$& 9.96e-07&$|$& 1.10e-07 & 3.87&$|$& 7.86&$|$& 39.67 \\
(100;3) & $2 \cdot 10^{-4}$ & $2 \cdot 10^{-5}$ & 0.123 & 6.28e-07&$|$& 9.97e-07&$|$& 1.12e-07 & 3.65&$|$& 7.69&$|$& 50.29 \\
  & $10^{-4}$ & $10^{-5}$ & 0.219 & 8.36e-07&$|$& 9.99e-07&$|$& 3.15e-07 & 5.09&$|$& 8.33&$|$& 67.48 \\
  & $\mathbf{10^{-4}}$ & $\mathbf{10^{-3}}$ & 0.238 & 4.52e-07&$|$& 9.95e-07&$|$& 5.04e-07 & 7.52&$|$& 8.28&$|$& 13.26 \\
\hline
SPX3a200 & $5 \cdot 10^{-4}$ & $5 \cdot 10^{-5}$ & 0.013 & 7.43e-07&$|$& 9.99e-07&$|$& 2.68e-08 & 10.03&$|$& 17.57&$|$& 188.39 \\
(200;3) & $2 \cdot 10^{-4}$ & $2 \cdot 10^{-5}$ & 0.081 & 7.06e-07&$|$& 1.00e-06&$|$& 4.34e-09 & 11.64&$|$& 18.23&$|$& 54.41 \\
  & $10^{-4}$ & $10^{-5}$ & 0.136 & 9.37e-07&$|$& 9.99e-07&$|$& 3.85e-07 & 24.68&$|$& 82.51&$|$& 107.52 \\
  & $\mathbf{10^{-4}}$ & $\mathbf{10^{-6}}$ & 0.134 & 5.56e-07&$|$& 9.98e-07&$|$& 8.30e-08 & 16.42&$|$& 22.51&$|$& 131.74 \\
\hline
SPX10b100 & $10^{-3}$ & $10^{-4}$ & 0.024 & 7.03e-07&$|$& 9.98e-07&$|$& 9.92e-07 & 9.62&$|$& 20.97&$|$& 1442.14 \\
(100;10) & $5 \cdot 10^{-4}$ & $5 \cdot 10^{-5}$ & 0.086 & 6.94e-07&$|$& 9.98e-07&$|$& 9.47e-07 & 11.87&$|$& 21.50&$|$& 1821.04 \\
  & $3 \cdot 10^{-4}$ & $3 \cdot 10^{-5}$ & 0.207 & 6.52e-07&$|$& 9.98e-07&$|$& 9.64e-07 & 13.10&$|$& 22.24&$|$& 1951.61 \\
  & $\mathbf{10^{-4}}$ & $\mathbf{10^{-3.5}}$ & 0.258 & 4.68e-07&$|$& 9.96e-07&$|$& 9.75e-07 & 20.18&$|$& 24.78&$|$& 1486.64 \\
\hline
SPX10b200 & $10^{-3}$ & $10^{-4}$ & 0.017 & 9.63e-07&$|$& 9.97e-07&$|$& 4.05e-05 & 42.34&$|$& 77.99&$|$& - \\
(200;10) & $5 \cdot 10^{-4}$ & $5 \cdot 10^{-5}$ & 0.064 & 9.62e-07&$|$& 9.98e-07&$|$& 8.94e-05 & 38.27&$|$& 81.43&$|$& - \\
  & $3 \cdot 10^{-4}$ & $3 \cdot 10^{-5}$ & 0.131 & 9.39e-07&$|$& 9.98e-07&$|$& 1.81e-04 & 50.29&$|$& 88.14&$|$& - \\
  & $\mathbf{10^{-4.5}}$ & $\mathbf{10^{-4}}$ & 0.241 & 7.49e-07&$|$& 9.98e-07&$|$& 1.70e-03 & 138.57&$|$& 183.88&$|$&- \\
\hline
SPX13c100 & $10^{-3}$ & $10^{-4}$ & 0.017 & 6.94e-07&$|$& 9.96e-07&$|$& 5.67e-07 & 12.66&$|$& 29.42&$|$& 1043.11 \\
(100;13) & $5 \cdot 10^{-4}$ & $5 \cdot 10^{-5}$ & 0.060 & 6.90e-07&$|$& 9.96e-07&$|$& 8.77e-07 & 15.12&$|$& 29.41&$|$& 2851.21 \\
  & $3 \cdot 10^{-4}$ & $3 \cdot 10^{-5}$ & 0.180 & 6.75e-07&$|$& 9.98e-07&$|$& 8.73e-07 & 18.53&$|$& 29.33&$|$& 3182.55 \\
  & $\mathbf{10^{-4}}$ & $\mathbf{10^{-3}}$ & 0.276 & 5.18e-07&$|$& 9.95e-07&$|$& 9.08e-07 & 29.35&$|$& 35.09&$|$& 1570.87 \\
\hline
SPX13c200 & $10^{-3}$ & $10^{-4}$ & 0.011 & 6.87e-07&$|$& 9.97e-07&$|$& 1.79e-04 & 49.28&$|$& 99.52&$|$& - \\
(200;13) & $5 \cdot 10^{-4}$ & $5 \cdot 10^{-5}$ & 0.047 & 7.58e-07&$|$& 9.97e-07&$|$& 6.95e-05 & 58.27&$|$& 102.86&$|$& - \\
  & $3 \cdot 10^{-4}$ & $3 \cdot 10^{-5}$ & 0.118 & 7.75e-07&$|$& 9.98e-07&$|$& 4.67e-04 & 59.43&$|$& 109.37&$|$& - \\
  & $\mathbf{10^{-4.5}}$ & $\mathbf{10^{-5}}$ & 0.279 & 8.56e-07&$|$& 9.98e-07&$|$& 4.99e-03 & 143.89&$|$& 234.70&$|$& - \\
\hline
SPX13d212 & $5 \cdot 10^{-4}$ & $5 \cdot 10^{-5}$ & 0.035 & 9.98e-07&$|$& 9.98e-07&$|$& 5.18e-05 & 56.56&$|$& 142.78&$|$& - \\
(212;13) & $2 \cdot 10^{-4}$ & $2 \cdot 10^{-5}$ & 0.166 & 9.20e-07&$|$& 9.99e-07&$|$& 5.09e-04 & 80.24&$|$& 155.26&$|$& - \\
  & $10^{-4}$ & $10^{-5}$ & 0.218 & 8.32e-07&$|$& 1.00e-06&$|$& 1.44e-03 & 105.51&$|$& 185.37&$|$& - \\
  & $\mathbf{10^{-4.5}}$ & $\mathbf{10^{-5}}$ & 0.268 & 5.64e-07&$|$& 1.00e-06&$|$& 6.05e-03 & 205.38&$|$& 710.61&$|$& - \\
\bottomrule

		\end{tabular}
\end{table}
In Table \ref{tab:22}, we compare the performance of {\sc Ssnpal}, ADMM, and MGL on four stock price data sets: SPX3a, SPX10b, SPX13c, and SPX13d. For each data set, we consider multiple sparsity levels (ranging from 1\% to 25\%) to evaluate how sensitive the methods are to the sparsity of the cMGGMs solutions.
The last parameter setting (shown in bold) is determined through 3-fold cross validation with turning grids $\mu_{1,\text{grid}}=\mu_{2,\text{grid}}=\left \{ 10^{-3}, 10^{-3.5},\right.\allowbreak\left. 10^{-4}, 10^{-4.5},10^{-5}, 10^{-5.5}, 10^{-6}\right \} $. 
As shown in the table, {\sc Ssnpal} consistently outperforms both ADMM and MGL across all cases. In addition, both {\sc Ssnpal} and ADMM are able to successfully solve all cases in under one hour, whereas MGL fails to converge for some cases within the same time limit. 
The significant  variation in computational time for MGL across problem dimensions indicates that its computational efficiency  is less stable than that of {\sc Ssnpal} and ADMM.
Overall, the numerical results demonstrate that the proposed {\sc Ssnpal} algorithm can solve  the cMGGMs  efficiently and robustly when applied to the stock price data sets.

\begin{figure}[htbp]
	\centering
	\hspace*{-1 em}
	\begin{minipage}[t]{0.37\textwidth}
		\centering
		\includegraphics[width=\linewidth]{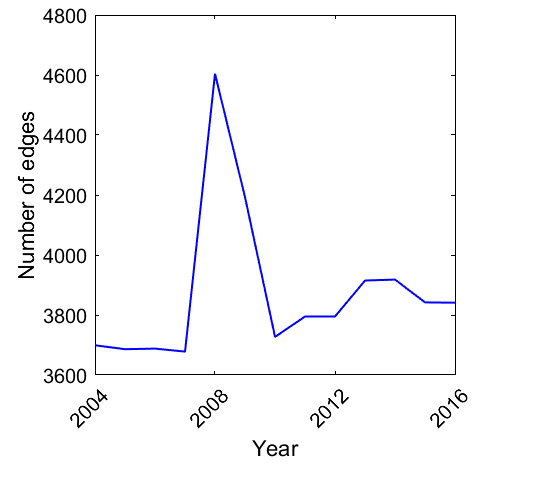}\\
		(a) 
	\end{minipage}
	\hspace{-2.5 em}
	\begin{minipage}[t]{0.37\textwidth}
		\centering
		\includegraphics[width=\linewidth]{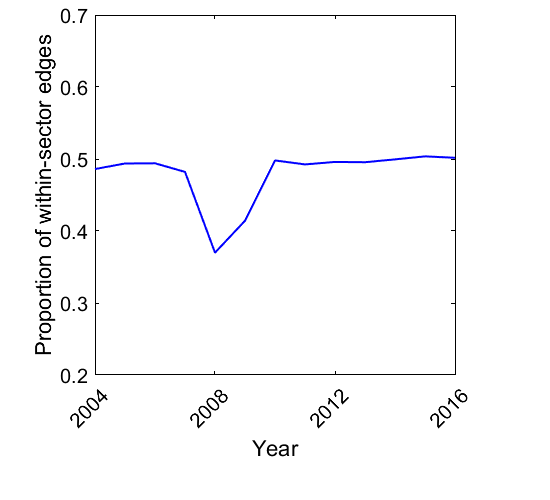}\\
		(b) 
	\end{minipage}
	\hspace{-2.5em}
	\begin{minipage}[t]{0.37\textwidth}
		\centering
		\includegraphics[width=\linewidth]{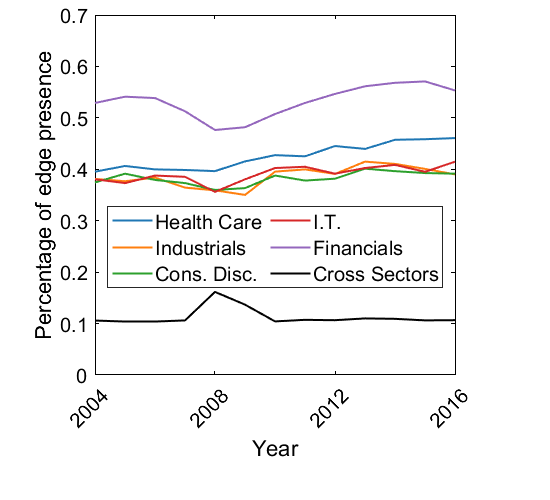}\\
		(c)
	\end{minipage}	
	\caption{Patterns of the estimated precision matrices for the stock price data set over the 9 year study period.  (a) total number of edges estimated to be nonzero over time; (b) proportion of within-sector edges among all detected edges over time; (c) percentage of edge presence for each sector and across sectors over time.}
	\label{fig:SP}
\end{figure}

Figure \ref{fig:SP} illustrates the sparse patterns of 13 estimated precision matrices from 2004 to
2016 on the SPX13d data set with  cross validation selected turning parameters $(\hat{\mu}_1, \hat{\mu}_2)=(10^{-4.5},10^{-5})$, $p=212$. 
Figure~\ref{fig:SP}(a) shows the number of edges in the fitted graph over time. The cMGGMs captures increased interactions among stocks during the financial crisis, with the number of edges peaking around 2008 and remaining above pre-crisis levels afterward.
Figure~\ref{fig:SP}(b) displays the proportion of within-sector edges among the total number of detected edges. During the entire time period,  this proportion decreased during the financial crisis due to increased cross-sector interaction, and eventually stabilized after 2010.
Figure~\ref{fig:SP}(c) shows the sector-wise percentage of presence of within-sector edges and the percentage of presence of cross-sector edges. 
As can be seen from this figure, the within-sector percentages are much higher than the cross-sector percentage, confirming that the cMGGMs can  capture the underlying sector structure. Moreover, the Financials sector exhibits the highest within-sector connectivity, indicating that the stocks in this sector have been consistently highly interacting with each other.
In summary, the graphs estimated by the cMGGMs offer clearer insights into the evolving interacting relationships among stocks and effectively reveal  underlying sector structure of stocks.

\subsubsection{ University webpages}
In this subsection, we compare the performance of {\sc Ssnpal}, ADMM, and MGL on the university webpages data set\footnote{Available at \url{ http://ana.cachopo.org/datasets-for-single-label-text-categorization}.}.  
For completeness, we briefly describe this data set, which has been discussed in 
 \cite{zhang2021efficient}.
The original data, collected from computer science departments of several universities in 1997, was manually categorized into seven classes: Student, Faculty, Course, Project, Staff, Department, and other.
Following the preprocessing approach in 
 \cite{cachopo2007improving}, we focus on a subset comprising the first four classes,  where standard stemming techniques were applied.  The data set was randomly partitioned into a training set (Webtrain), containing two-thirds of the webpages, and a testing set (Webtest), consisting of the remaining one-third.

\begin{table}[htbp]
	\footnotesize
	\caption{\footnotesize {Performance of {\sc Ssnpal}, ADMM, and MGL on university webpages data. In the table, ``a'' = {\sc Ssnpal}, ``b'' = ADMM,  and  ``c'' = MGL. Times are shown in seconds.}}\label{tab:23}
		\begin{tabular}{l@{\hspace{10pt}} l@{\hspace{5pt}} l@{\hspace{10pt}}l @{\hspace{10pt}}c@{\hspace{2pt}}c@{\hspace{2pt}}c@{\hspace{2pt}}c@{\hspace{2pt}}c@{\hspace{23pt}}c@{\hspace{2pt}}c@{\hspace{2pt}}c@{\hspace{2pt}}c@{\hspace{2pt}}c } \toprule
			\multirow{2}{*}{Problem} & 	
			\multirow{2}{*}{$\mu_1$} & 
			\multirow{2}{*}{$\mu_2$} & 
			\multirow{2}{*}{Density} & 
			
			\multicolumn{5}{c}{$R_{\text{KKT}}$} & 
			\multicolumn{5}{c}{Time} \\
			\cmidrule(r ){5-9} \cmidrule( ){10-14}  
			(p;K)& &&&a &$|$& b &$|$&c & a&$|$&b &$|$& c         \\

			\hline
		 	
			webtest100 & $10^{-2}$ & $10^{-3}$ & 0.014 & 6.40e-07&$|$& 1.00e-06&$|$& 5.80e-08 & 2.16&$|$& 6.07&$|$& 6.39 \\
(100;4) & $5 \cdot 10^{-3}$ & $5 \cdot 10^{-4}$ & 0.043 & 6.24e-07&$|$& 1.00e-06&$|$& 1.31e-08 & 2.00&$|$& 6.10&$|$& 25.40 \\
  & $10^{-3}$ & $10^{-4}$ & 0.211 & 6.00e-07&$|$& 1.00e-06&$|$& 6.57e-07 & 2.02&$|$& 6.14&$|$& 85.01 \\
  & $\mathbf{10^{-3}}$ & $\mathbf{10^{-3}}$ & 0.209 & 8.17e-07&$|$& 9.97e-07&$|$& 1.89e-07 & 3.81&$|$& 6.44&$|$& 58.14 \\
\hline
webtest200 & $5 \cdot 10^{-3}$ & $5 \cdot 10^{-4}$ & 0.028 & 7.04e-07&$|$& 9.94e-07&$|$& 1.92e-07 & 12.20&$|$& 16.89&$|$& 118.73 \\
(200;4) & $10^{-3}$ & $10^{-4}$ & 0.174 & 9.52e-07&$|$& 9.99e-07&$|$& 5.72e-07 & 8.11&$|$& 16.58&$|$& 538.20 \\
  & $5 \cdot 10^{-4}$ & $5 \cdot 10^{-5}$ & 0.327 & 8.10e-07&$|$& 9.99e-07&$|$& 9.71e-07 & 10.29&$|$& 16.47&$|$& 967.42 \\
  & $\mathbf{10^{-3}}$ & $\mathbf{10^{-2.5}}$ & 0.152 & 7.43e-07&$|$& 1.00e-06&$|$& 9.12e-07 & 27.27&$|$& 169.27&$|$& 393.76 \\
\hline
webtest300 & $5 \cdot 10^{-3}$ & $5 \cdot 10^{-4}$ & 0.013 & 6.77e-07&$|$& 9.98e-07&$|$& 7.17e-07 & 29.94&$|$& 54.07&$|$& 238.19 \\
(300;4) & $10^{-3}$ & $10^{-4}$ & 0.107 & 7.53e-07&$|$& 9.97e-07&$|$& 7.94e-07 & 18.57&$|$& 54.78&$|$& 1935.73 \\
  & $5 \cdot 10^{-4}$ & $5 \cdot 10^{-5}$ & 0.227 & 8.51e-07&$|$& 1.00e-06&$|$& 8.88e-07 & 26.73&$|$& 66.40&$|$& 3447.97 \\
  & $\mathbf{10^{-3}}$ & $\mathbf{10^{-3}}$ & 0.091 & 6.90e-07&$|$& 9.99e-07&$|$& 9.23e-07 & 30.26&$|$& 67.58&$|$& 602.23 \\
\hline
webtrain100 & $10^{-2}$ & $10^{-3}$ & 0.011 & 6.58e-07&$|$& 1.00e-06&$|$& 2.33e-07 & 5.36&$|$& 21.26&$|$& 51.42 \\
(100;4) & $5 \cdot 10^{-3}$ & $5 \cdot 10^{-4}$ & 0.027 & 6.58e-07&$|$& 9.98e-07&$|$& 5.17e-07 & 5.47&$|$& 20.05&$|$& 66.73 \\
  & $10^{-3}$ & $10^{-4}$ & 0.162 & 6.54e-07&$|$& 9.99e-07&$|$& 9.85e-07 & 5.34&$|$& 21.48&$|$& 436.41 \\
  & $\mathbf{10^{-3}}$ & $\mathbf{10^{-3.5}}$ & 0.164 & 6.47e-07&$|$& 9.98e-07&$|$& 9.23e-07 & 7.20&$|$& 21.02&$|$& 300.80 \\
\hline
webtrain200 & $5 \cdot 10^{-3}$ & $5 \cdot 10^{-4}$ & 0.013 & 8.74e-07&$|$& 1.00e-06&$|$& 4.62e-07 & 8.49&$|$& 33.06&$|$& 59.56 \\
(200;4) & $10^{-3}$ & $10^{-4}$ & 0.102 & 8.58e-07&$|$& 1.00e-06&$|$& 6.51e-07 & 8.65&$|$& 33.89&$|$& 417.93 \\
  & $5 \cdot 10^{-4}$ & $5 \cdot 10^{-5}$ & 0.204 & 8.29e-07&$|$& 9.97e-07&$|$& 8.08e-07 & 8.77&$|$& 33.38&$|$& 709.48 \\
  & $\mathbf{10^{-3.5}}$ & $\mathbf{10^{-4}}$ & 0.276 & 8.15e-07&$|$& 9.97e-07&$|$& 9.54e-07 & 11.28&$|$& 33.86&$|$& 670.88 \\
\hline
webtrain300 & $5 \cdot 10^{-3}$ & $5 \cdot 10^{-4}$ & 0.008 & 6.01e-07&$|$& 9.98e-07&$|$& 1.44e-07 & 18.94&$|$& 66.16&$|$& 175.02 \\
(300;4) & $10^{-3}$ & $10^{-4}$ & 0.073 & 7.48e-07&$|$& 9.99e-07&$|$& 6.98e-07 & 17.58&$|$& 69.89&$|$& 962.98 \\
  & $5 \cdot 10^{-4}$ & $5 \cdot 10^{-5}$ & 0.162 & 8.17e-07&$|$& 9.97e-07&$|$& 8.80e-07 & 20.29&$|$& 66.88&$|$& 1705.50 \\
  & $\mathbf{10^{-3.5}}$ & $\mathbf{10^{-4}}$ & 0.223 & 8.31e-07&$|$& 9.99e-07&$|$& 9.15e-07 & 24.53&$|$& 69.51&$|$& 1822.34 \\
\bottomrule

		\end{tabular}
\end{table}
 
Table~\ref{tab:23} reports the comparison results of three algorithms—{\sc Ssnpal}, ADMM, and MGL — on the university webpages data set with data dimension $p=100$, $p=200$, and $p=300$.
For each instance, the first three pairs of $(\mu_1, \mu_2)$ are manually selected to produce solutions with varying sparsity levels, while the final bold pair is obtained from  3-fold cross validation over the tuning grids $\mu_{1,\text{grid}} = \mu_{2,\text{grid}} = \left \{ 10^{-2}, 10^{-2.5}, 10^{-3}, 10^{-3.5}, 10^{-4}, 10^{-4.5}, 10^{-5}\right \}$.  
The numerical results presented in Table \ref{tab:23} indicate that {\sc Ssnpal} outperforms both ADMM and MGL in terms of computational efficiency for every examined webpages data set.
Moreover, {\sc Ssnpal} successfully solves all cases to the desired accuracy within one minute, whereas the runtime of MGL varies sharply with increasing problem dimensions.
This demonstrates the robustness and high efficiency of {\sc Ssnpal} in solving the cMGGMs when applied to the university webpages data set.

\subsubsection{20 newsgroups}

In this subsection, we conduct a comparative evaluation of {\sc Ssnpal}, ADMM, and MGL on the 20 newsgroups data. This data set consists  of documents nearly evenly distributed among 20 distinct newsgroups, each representing a specific topic. 
Some newsgroups exhibit strong topical similarity, such as comp.sys.ibm.pc.hardware and comp.sys.mac.hardware, whereas others are thematically unrelated, for example, misc.forsale and soc.religion.christian.

\begin{table}[htbp]
	\centering
	\caption{\footnotesize {Partition of 20 newsgroups by topics.}}
	\label{tb:5}
	\begin{tabularx}{0.8\textwidth}{>{\raggedright\arraybackslash}X}
		\toprule
		
		\textbf{NG1:}$\:\:\:\:$\begin{tabular}[t]{@{}l@{}}
			comp.graphics, 
			comp.os.ms-windows.misc, \\
			comp.sys.ibm.pc.hardware, 
			comp.sys.mac.hardware, 
			comp.windows.x
		\end{tabular}\\  \hline 
		\addlinespace
		
		\textbf{NG2:}$\:\:\:\:$rec.autos, 
		rec.motorcycles, 
		rec.sport.baseball, 
		rec.sport.hockey \\  \hline 
		\addlinespace
		
		\textbf{NG3:}$\:\:\:\:$sci.crypt, 
		sci.electronics, 
		sci.med, 
		sci.space \\  \hline 
		\addlinespace
		
		\textbf{NG4:}$\:\:\:\:$talk.politics.misc, 
		talk.politics.guns, 
		talk.politics.mideast \\  \hline 
		\addlinespace
		
		\textbf{NG5:}$\:\:\:\:$talk.religion.misc, 
		alt.atheism, 
		soc.religion.christian \\  \hline 
		\addlinespace
		
		\textbf{NG6:}$\:\:\:\:$misc.forsale \\
		\bottomrule
	\end{tabularx}
\end{table}
Table \ref{tb:5} enumerates the 20 newsgroups grouped by subject matter, as obtained from the publicly available repository\footnote{Available at \url{http://qwone.com/~jason/20Newsgroups/}.}.  
For  numerical experiments, we focus on the first five subgroups presented in Table  \ref{tb:5}—NG1 through NG5—which are expected to share common semantic structures. 
Each subgroup comprises multiple classes, and we employ the clustered structure to jointly estimate the precision matrices corresponding to different classes within each subgroup.
A preprocessed version of the 20 newsgroups data, formatted for convenient use in MATLAB, is available from the same source. The downloaded package includes both training and testing sets. Following the procedure described in \cite{guo2011joint}, we generate sample covariance matrices with the number of features set to $p=300$. 

%
%
%
%
%
%
%
%
%
\begin{table}[htbp]
	\footnotesize
	 \caption{\footnotesize {Performance of {\sc Ssnpal}, ADMM, and MGL on 20 newsgroups data. In the table, ``a'' = {\sc Ssnpal}, ``b'' = ADMM,  and  ``c'' = MGL. Times are shown in seconds.}}\label{tab:24}
	\begin{tabular}{l@{\hspace{10pt}} l@{\hspace{5pt}} l@{\hspace{10pt}}l @{\hspace{10pt}}c@{\hspace{2pt}}c@{\hspace{2pt}}c@{\hspace{2pt}}c@{\hspace{2pt}}c@{\hspace{25pt}}c@{\hspace{2pt}}c@{\hspace{2pt}}c@{\hspace{2pt}}c@{\hspace{2pt}}c }
		 
		\toprule
		\multirow{2}{*}{Problem} &

		\multirow{2}{*}{$\mu_1$} & 
		\multirow{2}{*}{$\mu_2$} & 
		\multirow{2}{*}{Density} & 
		
		\multicolumn{5}{c}{$R_{\text{KKT}}$} & 
		\multicolumn{5}{c}{Time} \\
		\cmidrule(r ){5-9} \cmidrule( ){10-14}  
		(p;K)& &&&a &$|$& b &$|$&c & a&$|$&b &$|$& c         \\
		
		\hline
		 	
		NG1test & $5 \cdot 10^{-3}$ & $5 \cdot 10^{-4}$ & 0.018 & 9.79e-07&$|$& 1.00e-06&$|$& 4.59e-06 & 23.63&$|$& 88.44&$|$& - \\
(300;5) & $10^{-3}$ & $10^{-4}$ & 0.088 & 8.31e-07&$|$& 9.99e-07&$|$& 3.64e-05 & 41.49&$|$& 96.26&$|$& - \\
  & $5 \cdot 10^{-4}$ & $5 \cdot 10^{-5}$ & 0.171 & 8.98e-07&$|$& 9.98e-07&$|$& 1.65e-04 & 41.50&$|$& 116.21&$|$& - \\
  & $\mathbf{10^{-3}}$ & $\mathbf{10^{-3}}$ & 0.107 & 7.09e-07&$|$& 9.99e-07&$|$& 4.39e-06 & 109.13&$|$& 427.39&$|$&- \\
\hline
NG2test & $5 \cdot 10^{-3}$ & $5 \cdot 10^{-4}$ & 0.003 & 6.57e-07&$|$& 1.00e-06&$|$& 8.88e-07 & 38.84&$|$& 555.32&$|$& 476.73 \\
(300;4) & $10^{-3}$ & $10^{-4}$ & 0.057 & 6.49e-07&$|$& 1.00e-06&$|$& 8.48e-07 & 39.46&$|$& 570.29&$|$& 3500.47 \\
  & $5 \cdot 10^{-4}$ & $5 \cdot 10^{-5}$ & 0.119 & 6.40e-07&$|$& 1.00e-06&$|$& 1.68e-05 & 39.64&$|$& 580.16&$|$& - \\
  & $\mathbf{10^{-3.5}}$ & $\mathbf{10^{-3.5}}$ & 0.163 & 9.57e-07&$|$& 1.00e-06&$|$& 4.06e-06 & 96.09&$|$& 562.02&$|$& - \\
\hline
NG3test & $5 \cdot 10^{-3}$ & $5 \cdot 10^{-4}$ & 0.005 & 9.48e-07&$|$& 9.96e-07&$|$& 5.80e-07 & 17.65&$|$& 37.34&$|$& 283.42 \\
(300;4) & $10^{-3}$ & $10^{-4}$ & 0.071 & 9.25e-07&$|$& 9.97e-07&$|$& 4.69e-06 & 31.53&$|$& 39.28&$|$& - \\
  & $5 \cdot 10^{-4}$ & $5 \cdot 10^{-5}$ & 0.140 & 9.03e-07&$|$& 1.00e-06&$|$& 2.64e-05 & 18.43&$|$& 40.72&$|$& - \\
  & $\mathbf{10^{-3.5}}$ & $\mathbf{10^{-3.5}}$ & 0.194 & 6.73e-07&$|$& 9.99e-07&$|$& 1.73e-06 & 79.03&$|$& 168.56&$|$& - \\
\hline
NG4test & $5 \cdot 10^{-3}$ & $5 \cdot 10^{-4}$ & 0.029 & 7.44e-07&$|$& 9.99e-07&$|$& 9.86e-07 & 32.41&$|$& 148.23&$|$& 987.48 \\
(300;3) & $10^{-3}$ & $10^{-4}$ & 0.115 & 7.40e-07&$|$& 1.00e-06&$|$& 7.31e-06 & 24.92&$|$& 146.80&$|$& - \\
  & $5 \cdot 10^{-4}$ & $5 \cdot 10^{-5}$ & 0.235 & 6.95e-07&$|$& 9.94e-07&$|$& 2.26e-05 & 18.33&$|$& 31.57&$|$& - \\
  & $\mathbf{10^{-3}}$ & $\mathbf{10^{-3.5}}$ & 0.111 & 7.27e-07&$|$& 9.99e-07&$|$& 2.65e-06 & 27.73&$|$& 144.10&$|$& - \\
\hline
NG5test & $5 \cdot 10^{-3}$ & $5 \cdot 10^{-4}$ & 0.021 & 7.98e-07&$|$& 1.00e-06&$|$& 9.29e-07 & 19.59&$|$& 161.18&$|$& 496.22 \\
(300;3) & $10^{-3}$ & $10^{-4}$ & 0.117 & 7.73e-07&$|$& 1.00e-06&$|$& 1.90e-06 & 15.39&$|$& 164.30&$|$& - \\
  & $5 \cdot 10^{-4}$ & $5 \cdot 10^{-5}$ & 0.236 & 7.43e-07&$|$& 1.00e-06&$|$& 9.09e-06 & 18.17&$|$& 161.33&$|$& - \\
  & $\mathbf{10^{-3}}$ & $\mathbf{10^{-3}}$ & 0.111 & 7.06e-07&$|$& 1.00e-06&$|$& 9.31e-07 & 31.02&$|$& 160.73&$|$& 1995.74 \\
\hline
NG1train & $5 \cdot 10^{-3}$ & $5 \cdot 10^{-4}$ & 0.013 & 8.07e-07&$|$& 1.00e-06&$|$& 8.06e-07 & 47.54&$|$& 556.98&$|$& 3569.54 \\
(300;5) & $10^{-3}$ & $10^{-4}$ & 0.077 & 7.79e-07&$|$& 1.00e-06&$|$& 4.61e-05 & 51.63&$|$& 584.26&$|$& - \\
  & $5 \cdot 10^{-4}$ & $5 \cdot 10^{-5}$ & 0.127 & 7.67e-07&$|$& 1.00e-06&$|$& 9.91e-05 & 55.91&$|$& 573.00&$|$& - \\
  & $\mathbf{10^{-3.5}}$ & $\mathbf{10^{-4}}$ & 0.173 & 7.63e-07&$|$& 1.00e-06&$|$& 9.45e-05 & 72.67&$|$& 572.72&$|$& - \\
\hline
NG2train & $5 \cdot 10^{-3}$ & $5 \cdot 10^{-4}$ & 0.004 & 9.52e-07&$|$& 1.00e-06&$|$& 9.88e-07 & 34.22&$|$& 142.14&$|$& 602.04 \\
(300;4) & $10^{-3}$ & $10^{-4}$ & 0.066 & 9.31e-07&$|$& 1.00e-06&$|$& 5.07e-05 & 33.61&$|$& 160.89&$|$& - \\
  & $5 \cdot 10^{-4}$ & $5 \cdot 10^{-5}$ & 0.110 & 9.84e-07&$|$& 9.99e-07&$|$& 1.02e-04 & 31.99&$|$& 152.27&$|$& - \\
  & $\mathbf{10^{-3.5}}$ & $\mathbf{10^{-4}}$ & 0.147 & 9.81e-07&$|$& 9.99e-07&$|$& 1.04e-04 & 41.27&$|$& 155.54&$|$& - \\
\hline
NG3train & $5 \cdot 10^{-3}$ & $5 \cdot 10^{-4}$ & 0.009 & 5.98e-07&$|$& 9.99e-07&$|$& 6.26e-07 & 28.16&$|$& 124.75&$|$& 757.02 \\
(300;4) & $10^{-3}$ & $10^{-4}$ & 0.085 & 6.07e-07&$|$& 1.00e-06&$|$& 2.80e-05 & 27.38&$|$& 134.62&$|$& - \\
  & $5 \cdot 10^{-4}$ & $5 \cdot 10^{-5}$ & 0.136 & 6.26e-07&$|$& 9.99e-07&$|$& 5.50e-05 & 40.96&$|$& 136.00&$|$& - \\
  & $\mathbf{10^{-3.5}}$ & $\mathbf{10^{-4.5}}$ & 0.217 & 6.94e-07&$|$& 9.98e-07&$|$& 1.06e-04 & 54.31&$|$& 140.86&$|$& - \\
\hline
NG4train & $5 \cdot 10^{-3}$ & $5 \cdot 10^{-4}$ & 0.031 & 6.12e-07&$|$& 1.00e-06&$|$& 2.00e-05 & 39.37&$|$& 1175.93&$|$& - \\
(300;3) & $10^{-3}$ & $10^{-4}$ & 0.108 & 8.52e-07&$|$& 1.00e-06&$|$& 1.04e-04 & 23.62&$|$& 258.95&$|$& - \\
  & $5 \cdot 10^{-4}$ & $5 \cdot 10^{-5}$ & 0.207 & 8.39e-07&$|$& 1.00e-06&$|$& 1.45e-04 & 22.16&$|$& 260.03&$|$& - \\
  & $\mathbf{10^{-3}}$ & $\mathbf{10^{-3.5}}$ & 0.108 & 9.71e-07&$|$& 1.00e-06&$|$& 8.33e-05 & 40.15&$|$& 1200.14&$|$& - \\
\hline
NG5train & $5 \cdot 10^{-3}$ & $5 \cdot 10^{-4}$ & 0.011 & 8.73e-07&$|$& 9.98e-07&$|$& 1.87e-07 & 10.69&$|$& 43.76&$|$& 51.22 \\
(300;3) & $10^{-3}$ & $10^{-4}$ & 0.107 & 6.15e-07&$|$& 9.98e-07&$|$& 3.13e-07 & 13.06&$|$& 45.83&$|$& 398.86 \\
  & $5 \cdot 10^{-4}$ & $5 \cdot 10^{-5}$ & 0.184 & 6.72e-07&$|$& 9.99e-07&$|$& 7.33e-07 & 16.30&$|$& 46.12&$|$& 1168.90 \\
  & $\mathbf{10^{-3}}$ & $\mathbf{10^{-3}}$ & 0.116 & 8.60e-07&$|$& 9.98e-07&$|$& 6.72e-07 & 32.36&$|$& 71.14&$|$& 420.67 \\
\bottomrule

	\end{tabular}	
\end{table}
Table~\ref{tab:24} presents the comparison results of {\sc Ssnpal}, ADMM, and MGL on the 20 newsgroups data set.
In each experiment, the pair of parameters $(\mu_1, \mu_2)$ (bold numbers in the table) is selected via  cross validation over the tuning grids $\mu_{1,\text{grid}}=\mu_{2,\text{grid}}=\{10^{-3},\allowbreak
10^{-3.5},\allowbreak
10^{-4},\allowbreak
10^{-4.5},\allowbreak
10^{-5},\allowbreak
10^{-5.5},\allowbreak
10^{-6}\}$.
It is clear from Table \ref{tab:24} that  {\sc Ssnpal} consistently outperforms both ADMM and MGL across all tested cases, which demonstrates its superior computational efficiency in solving the cMGGMs.
Notably, for some difficult instances, MGL requires over an hour to achieve the desired accuracy, whereas {\sc Ssnpal} typically converges within two minutes and demonstrates a clear speed advantage over ADMM.
Overall, these numerical results further confirm the robustness and high efficiency of {\sc Ssnpal} in solving the cMGGMs based on the 20 newsgroups data.
\section{Conclusion}\label{conclusion}
In this paper, we have proposed an inexact proximal augmented Lagrangian method to solve the dual formulation of cMGGMs.
Each subproblem within the framework is efficiently handled by a semismooth Newton method, which makes full use of second order information for fast and stable convergence.
The convergence properties of the {\sc Ssnpal} algorithm were shown to hold automatically under some standard conditions. 
Furthermore, we derived a necessary and sufficient condition under which the cMGGMs solution admits a block diagonal structure for an arbitrary number of graphs, and developed an effective screening rule to accelerate the estimation of large scale multiple graphs.
Extensive numerical experiments on both synthetic and real data sets demonstrate the superior efficiency and robustness of {\sc Ssnpal} compared with the widely used ADMM and the proximal Newton–type method 
 \cite{yang2015fused}.

 \bmhead{Acknowledgements}
 The research of Yong-Jin Liu was funded by National Key Research and Development Program of China (2025YFA1016901), National Natural Science Foundation of China (Grant No. 12271097), Key Program of National Science Foundation of Fujian Province of China (Grant No. 2023J02007), Central Guidance on Local Science and Technology Development Fund of Fujian Province (Grant No. 2023L3003).



\bibliography{references}  

\end{document}